\documentclass[12pt]{article}
\usepackage{fullpage}
\usepackage{amssymb}
\usepackage{amsthm}
\usepackage{graphicx}
\usepackage[all,2cell]{xy}
\newtheorem{theorem}{Theorem}[section]
\newtheorem{corollary}[theorem]{Corollary}

\newtheorem{lemma}[theorem]{Lemma}

\newtheorem{proposition}[theorem]{Proposition}

\begin{document}
\title{An analogue of Bott's theorem for Schubert varieties-related to torus semi-stable points} 
\author{S. Senthamarai Kannan}
\maketitle

Chennai Mathematical Institute, Plot H1, SIPCOT IT Park, 

Siruseri, Kelambakkam  603103, Tamilnadu, India.

E-mail:kannan@cmi.ac.in.

\begin{abstract}
Let $G$ be a simple, simply connected  
algebraic group over the field $\mathbb{C}$ of complex numbers. 
Let $B$ be a Borel subgroup of $G$ containing a maximal 
torus $T$ of $G$. Let $\mathcal{T}_{G/B}$ denote the tangent bundle of the 
flag variety $G/B$. Let $\tau$ be an element of the Weyl group $W$ and 
let $X(\tau)$ be the Schubert variety corresponding to $\tau$.

In this paper, we prove the following:

If $G$ is simply laced, then, we have 

\begin{enumerate}
\item
$H^{i}(X(\tau), \mathcal{L}_{G/B})=(0)$ for every $i\geq 1$. 
\item
$H^{0}(X(\tau) , \mathcal{L}_{G/B})$ is the adjoint representation 
$\mathfrak{g}$ of $G$ if and only if the set of semi-stable points 
$X(\tau^{-1})_{T}^{ss}(\mathcal{L}_{\alpha_{0}})$ with respect to the line 
bundle associated to the highest root $\alpha_{0}$ is non-empty. 
\end{enumerate}

If $G$ is not simply laced, then, we have 

\begin{enumerate}
\item
$H^{i}(X(\tau), \mathcal{L}_{G/B})=(0)$ for every $i\geq 1$. 
\item
The adjoint representation $\mathfrak{g}$ of $G$ is a $B$-submodule 
of $H^{0}(X(\tau) , \mathcal{L}_{G/B})$ 
if and only if the set of semi-stable points 
$X(\tau^{-1})_{T}^{ss}(\mathcal{L}_{\alpha_{0}})$ with respect to the line 
bundle associated to the highest root $\alpha_{0}$ is non-empty. 
\end{enumerate}
\end{abstract}


\section{Introduction}
In [3], Bott proved that for any semisimple algebraic group $G$ over the 
field of complex numbers, for any Borel sub group $B$ of $G$, all the 
higher cohomologies $H^{i}(G/B, \mathcal{T}_{G/B})$ 
with respect to the tangent bundle $\mathcal{T}_{G/B}$ on the flag variety 
vanishes. He further showed that the $G$- module of global sections 
$H^{0}(G/B, \mathcal{T}_{G/B})$ is the adjoint representation $\mathfrak{g}$ 
of $G$. 

It is a natural question to ask for which Schubert vaeriety $X(\tau)$ in the 
flag variety $G/B$, $H^{i}(X(\tau), \mathcal{T}_{G/B})$ with respect to 
restriction of the tangent bundle $\mathcal{T}_{G/B}$ on the flag variety 
to $X(\tau)$ vanishes and that the $B$- module of global sections 
$H^{0}(X(\tau), \mathcal{T}_{G/B})$ is the adjoint representation 
$\mathfrak{g}$ of $G$.

The tangent space of $idB$ in $G/B$ as a $T$ module is a direct sum
of weight spaces each of which is not dominant except the highest short root 
and highest long root. 

There are interesting and important results have been obtatained for line 
bundles corresponding to non dominant characters on Schubert varieties.
We refer to [1], [4] and [8] for some of the results. 

We may also refer to [2] and [9] for recent developments.

However, we do not seem to have a precise answer in the literature
for the above mentioned question.

Therefore, this question is of importance in relation to Schubert varieties.

The aim of this paper is to give a necessary and suffiicient condition on the 
Schubert varieties $X(\tau)$ in the simply laced flag variety $G/B$ 
for which the above question has an affirmative answer. 

We now proceed with notation before we describe our result.

The following notation will be maintained throughout this paper except in few places in section 3 where we prove some basic lemmas for algebraic groups 
over algebraically closed fields of arbitrary characteristic.

Let $\mathbb{C}$ denote the field of complex numbers. Let  
$G$ a simple, simply connected algebraic group 
over $\mathbb{C}$.  We fix a maximal torus $T$ of $G$ and 
let $X(T)$ denote the set of characters
of $T$. Let $W = N(T)/T$ denote the Weyl group of $G$ with respect to $T$.
Let $R$ denote the set of roots of $G$ with respect to $T$.

Let $R^{+}$ denote the set of positive roots. Let $B^{+}$ be the Borel 
sub group of $G$ contatining $T$ with respect to $R^{+}$.
Let $S = \{\alpha_1,\ldots,\alpha_l\}$ denote the set of simple roots in
$R^{+}$. Here $l$ is the rank of $G$. 
Let $B$ be the Borel subgroup of $G$ containing $T$ with respect to the 
set of negative roots $R^{-} =-R^{+}$.

For $\beta \in R^+$ we also
use the notation $\beta > 0$.  
The simple reflection in the Weyl group corresponding to $\alpha_i$ is denoted
by $s_{\alpha_i}$.  

Let $\mathfrak{g}$ denote the Lie algebra 
of $G$. Let $\mathfrak{h}$ be the Lie algebra of $T$. Let $\mathfrak{b}$ 
be the Lie algebra of $B$.

We have $X(T)\bigotimes \mathbb{R}=(\mathfrak{h}_{\mathbb{R}})^{*}$,
the dual of the real form of $\mathfrak{h}$.

The positive definite $W$-invariant form on 
$(\mathfrak{h}_{\mathbb{R}})^{*}$ induced by the Killing form of the 
Lie algebra $\mathfrak{g}$ of $G$ is denoted by $(~,~)$. 
We use the notation $\left< ~,~ \right>$ to
denote $\langle \nu, \alpha \rangle  = \frac{2(\nu,
\alpha)}{(\alpha,\alpha)}$. 

Let $x_{\alpha}, y_{\alpha}, \alpha \in
R^+, h_{\alpha_i}, \alpha_i \in S, $ denote a Chevalley basis of the Lie
algebra of $G$. 

We denote by $\mathfrak{g}_{\alpha}$ (resp. $\mathfrak{g}_{-\alpha}$)
the one dimensional root subspace of $\mathfrak{g}$ spanned by $x_{\alpha}$
(resp. $y_{\alpha}$).

Let $sl_{2,\alpha}$ denote the $3$ dimensional Lie sub algebra of $\mathfrak{g}$
generated by $x_{\alpha}$, and $y_{\alpha}$.

Let $\leq$ denote the partial order on $X(T)$ given by $\mu\leq \lambda$
if $\lambda-\mu$ is a non negative integral linear combination of simple 
roots.

We denote by $X(T)^+$ the set of dominant characters of 
$T$ with respect to $B^{+}$. Let $\rho$ denote the half sum of all 
positive roots of $G$ with respect to $T$ and $B^{+}$.

For any simple root $\alpha$, we denote the fundamental weight
corrsponding to $\alpha$  by $\omega_{\alpha}$. 

For $w \in W$ let
$l(w)$ denote the length of $w$. We define the 
dot action by $w\cdot\lambda=w(\lambda+\rho)-\rho$.

Let $\alpha_{0}$ denote the highest root.

We set $R^{+}(w):=\{\beta\in R^{+}:w(\beta)\in -R^{+}\}$.

Let $w_0$ denote the longest
element of the Weyl group $W$.

For $w \in W$, let $X(w):=\overline{BwB/B}$ denote the Schubert variety in
$G/B$ corresponding to $w$. 

Consider the $T$ action of $G/B$. Schubert vaerieties $X(w)$ are stable 
under $T$. Let $\lambda$ be a dominant character of $T$. We denote by $\mathcal{L}_{\lambda}$ denote the line bundle on $G/B$ corresponding to the character
$\lambda$ of $B$. We denote the restriction of the line bundle 
$\mathcal{L}_{\lambda}$ to $X(w)$ as well by 
$\mathcal{L}_{\lambda}$. 

We denote by $X(w)_{T}^{ss}(\mathcal{L}_{\lambda})$ the set of all semi-stable 
points of $X(w)$ with respect to the line bundle $\mathcal{L}_{\lambda}$
for the action of $T$. 

So, inparticular, we have semi-stable points 
$X(w)_{T}^{ss}(\mathcal{L}_{\alpha_{0}})$  with respect to the line bundle 
$\mathcal{L}_{\alpha_{0}}$ corresponding to the highest root $\alpha_{0}$.

In this paper, we prove the following theorem for simple, simply connected and 
simply laced algebraic groups.

{\bf Theorem A}

Let $G$ be a simple, simply connected and simply laced algebraic group over 
$\mathbb{C}$. 
Let $\tau \in W$. Then, we have
\begin{enumerate}
\item
$H^{i}(X(\tau), \mathcal{T}_{G/B})=(0)$ for every 
$i\geq 1$. \\
\item
$H^{0}(X(\tau) , \mathcal{T}_{G/B})$ is the adjoint representation $\mathfrak{g}$of  $G$ if and only if the set of semi-stable points $X(\tau^{-1})_{T}^{ss}(\mathcal{L}_{\alpha_{0}})$ 
is non-empty. \\
\end{enumerate}

We also  prove that 

{\bf Theorem B}

Let $G$ be simple, simply connected but not simply laced algebraic group over 
$\mathbb{C}$. 
Let $\tau \in W$. Then, we have

\begin{enumerate}
\item
$H^{i}(X(\tau), \mathcal{T}_{G/B})=(0)$ for every 
$i\geq 2$. 
\item
The adjoint representation $\mathfrak{g}$ is a $B$-submodule of 
$H^{0}(X(\tau) , \mathcal{T}_{G/B})$  if and 
only if the set of semi-stable points 
$X(\tau^{-1})_{T}^{ss}(\mathcal{L}_{\alpha_{0}})$ is non-empty.
\end{enumerate}

The organisation of the paper is as follows:

Section 2 consists of preliminaries from [5], [6] and [7]. 
In section 3, we prove theorem A. In section 4, we apply theorem A to 
certain Schubert varieties related to maximal parabolic subgroups 
of $G$. For precise statement, see theorem(4.2).

In section 5, we obtain the following theorem on the cohomology modules
$H^{i}(X(c), \mathcal{L}_{c^{-1}\cdot 0})$ of the line bundle $\mathcal{L}_{c^{-1}\cdot 0}$ on the Schubert variety $X(c)$ corresponding to a Coxeter element $c$ of $W$. We use theorem A in proving this theorem.

{\bf Theorem C}
\begin{enumerate}
\item
Let $\tau\in W$. The cohomology module $H^{l(\tau)}(X(\tau), \mathcal{L}_{\tau^{-1}\cdot 0})$ 
is the one dimensional trivial representation of $B$. 
\item
Let $c$ be a Coxeter element of $W$. Then, 
$H^{i}(X(c), \mathcal{L}_{c^{-1}\cdot 0})$ is zero for every $i\neq l(c)$ if and only if both 
$X(c)_{T}^{ss}(\mathcal{L}_{\alpha_{0}})$ and 
$X(c^{-1})_{T}^{ss}(\mathcal{L}_{\alpha_{0}})$ are non-empty.
\end{enumerate}

For a precise detail with notation, see theorem(5.7).

In section 6, we prove theorem B.

\section{Preliminaries}
\label{prelim}

We denote by $U$ the unipotent radical of $B$. We denote by
$P_{\alpha}$ the minimal parabolic subgroup of $G$ containing $B$ and
$s_{\alpha}$.  Let $L_{\alpha}$ denote the Levi subgroup of
$P_{\alpha}$ containing $T$. We denote by $B_{\alpha}$ the
intersection of $L_{\alpha}$ and $B$. Then $L_{\alpha}$ is the product
of $T$ and a homomorphic image $G_{\alpha}$ of $SL(2)$ via a homomorphism
$\psi: SL(2)\longrightarrow L_{\alpha}$.
(cf. [7, II , 1.1.4]). 

We make use of following points in computing cohomologies. 

{\it  Since $G$ is simply connected, the morphism 
$\psi: SL(2)\longrightarrow G_{\alpha}$ is an isomorphism, 
and hence $\psi:SL(2)\longrightarrow L_{\alpha}$  is injective.
We denote this copy of $SL(2)$ in $L_{\alpha}$  by $SL(2,\alpha)$ 
We denote by $B^{\prime}_{\alpha}$ the intersection of $B_{\alpha}$
and $SL(2,\alpha)$ in $L_{\alpha}$.

We also note that the morphism 
$SL(2, \alpha)/B_{\alpha}^{\prime}\hookrightarrow  L_{\alpha}/B_{\alpha}$ 
induced by $\psi$ is an isomorphism.}

{\it Since $L_{\alpha}/B_{\alpha}\hookrightarrow P_{\alpha}/B$ 
is an isomorphism, to compute the cohomology $H^{i}(P_{\alpha}/B, V)$
for any $B$- module $V$, we treat $V$ as a $B_{\alpha}$- module 
and we compute $H^{i}(L_{\alpha}/B_{\alpha}, V)$ } 

Given a $w \in W$ the closure in $G/B$ of the $B$ orbit of the coset
$wB$ is the Schubert variety corresponding to $w$, and is denoted by
$X(w)$. We recall some basic facts and results about Schubert
varieties. A good reference for all this is the book by Jantzen. (cf
[7, II, Chapter 14 ] ).

Let $w = s_{\alpha_{i_{1}}}s_{\alpha_{i_{2}}}\ldots s_{\alpha_{i_{n}}}$ be a reduced
expression for $w \in W$. Define 
\[
Z(w) = \frac {P_{\alpha_{i_[1}}\times P_{\alpha_{i_{2}}} \times \ldots \times 
P_{\alpha_{i_{n}}}}{B \times \ldots
\times B},
\]
where the action of $B \times \ldots \times B$ on $P_{\alpha_{i_{1}}} \times P_{\alpha_{i_{2}}}
\times \ldots \times P_{\alpha_{i_{n}}}$ is given by $(p_1, \ldots , p_n)(b_1, \ldots
, b_n) = (p_1 \cdot b_1, b_1^{-1} \cdot p_2 \cdot b_2, \ldots
,b^{-1}_{n-1} \cdot p_n \cdot b_n)$, $ p_j \in P_{\alpha_{i_{j}}}$, $b_j \in B$. 
We denote by $\phi_w$ the birational surjective morphism
$
\phi_w : Z(w) \longrightarrow X(w)$.

We note that for each reduced expression for $w$, $Z(w)$ is smooth, however, 
$Z(w)$ may not be independent of a reduced expression.

Let $f_n : Z(w) \longrightarrow Z(ws_{\alpha_{n}})$ denote the map induced by the
projection $P_{\alpha_1} \times P_{\alpha_2} \times \ldots \times 
P_{\alpha_n} \longrightarrow P_{\alpha_1} \times P_{\alpha_2}
\times \ldots \times P_{\alpha_{n-1}}$. Then we observe that $f_n$ is a $ P_{\alpha_n}/B
\simeq {\bf P}^{1}$-fibration.

Let $V$ be a $B$-module. Let ${\mathcal L}_w(V)$ denote the pull back to $X(w)$ of the homogeneous vector bundle on $G/B$ associated to $V$. 
{\em By abuse of notation} we denote the pull back of ${\mathcal L}_w(V)$ to $Z(w)$ 
also by ${\mathcal L}_w(V)$, when there is no cause for confusion. 
Then, for $i \geq 0$, we have the following isomorphisms of
$B$-linearized sheaves
\[
R^{i}{f_{n}}_{*}{\mathcal L}_w(V) = {\mathcal
L}_{ws_{\alpha_{n}}}(H^{i}(P_{\alpha_n}/B, {\mathcal L}_w(V)).
\]
This together with easy applications of Leray spectral sequences is
the constantly used tool in what follows. We term this the {\em
descending 1-step construction}. 

We also have the {\em ascending 1-step construction} which too is used
extensively in what follows sometimes in conjunction with the
descending construction. We recall this for the convenience of the
reader.

Let the notations be as above and write $\tau = s_{\gamma}w$, with 
$l(\tau) = l(w) +1$, for some simple root $\gamma$.  Then we
have an induced morphism
\[
g_1: Z(\tau) \longrightarrow P_{\gamma}/B \simeq {\bf P}^1,
\]
with fibres given by $Z(w)$. Again, by an application of the Leray
spectral sequences together with the fact that the base is a ${\bf
P}^1$, we obtain for every $B$-module $V$ the following exact sequence
of $P_{\gamma}$-modules:
$$
(0) \longrightarrow H^{1}(P_{\gamma}/B, R^{i-1}{g_{1}}_{*}{\mathcal L}_w(V)) 
\longrightarrow 
H^{i}(Z(\tau) , {\mathcal L}_{\tau}(V)) \longrightarrow
H^{0}(P_{\gamma}/B , R^{i}{g_{1}}_{*}{\mathcal L}_w(V) ) \longrightarrow (0).$$

This short exact sequence of $B$-modules will be used frequently in this paper.
So, we denote this short exact sequence by {\it SES} when ever this is being 
used.

We also recall the following well-known isomorphisms:
\begin{itemize}

\item ${\phi_w}_*{\mathcal O}_{Z(w)} = {\mathcal O}_{X(w)}$.

\item $R^{q}{\phi_w}_*{\mathcal O}_{Z(w)} = 0$ for $q > 0$.

\end{itemize}

This together with [7, II. 14.6] implies that we may use the
Bott-Samelson schemes $Z(w)$ for the computation and study of all the
cohomology modules $H^{i}(X(w) , {\mathcal L}_w(V))$. Henceforth in this paper 
we shall use the Bott-Samelson schemes and their
cohomology modules in all the computations.

{\it Simplicity of Notation}
If $V$ is a $B$-module and ${\mathcal L}_w(V)$ is the induced
  vector bundle on $Z(w)$ we denote the cohomology modules
  $H^{i}(Z(w) , {\mathcal L}_w(V))$ by $H^{i}(w ,V)$.

In particular, if $\lambda$ is a character of $B$ we 
denote the cohomology modules $H^{i}(Z(w) , {\mathcal L}_{\lambda} )$ by
$H^i(w, \lambda)$.

\subsubsection {Some constructions from Demazure's paper}
We recall briefly two exact sequences from [5] that Demazure used in his short
proof of the Borel-Weil-Bott theorem (cf. [3] ). We use the same 
notation as in [5]. In the rest of the paper these sequences are 
referred to as Demazure exact sequences.

Let $\alpha$ be a simple root and let $\lambda \in X(T)$ be a weight
such that $\langle \lambda , \alpha \rangle  \geq 0$. For such a 
$\lambda$, we denote by $V_{\lambda,\alpha}$ the module 
$H^0(P_{\alpha}/B, \lambda)$ . Let $\mathbb{C}_{\lambda}$ denote the one 
dimensional $B$- module.

Here, we recall the following lemma due to Demazure on a short exact sequence of $B$ - modules: (to obtain the second sequence we need to assume that
$\langle \lambda , \alpha \rangle \geq 2$).

\begin{lemma}

$$
\begin{array}{l}
\mbox{$(0) \longrightarrow K \longrightarrow V_{\lambda,\alpha} \longrightarrow \mathbb{C}_\lambda \longrightarrow (0)$}.\\
\mbox{$(0) \longrightarrow \mathbb{C}_{s_{\alpha}(\lambda)} \longrightarrow K \longrightarrow V_{\lambda-\alpha,\alpha}
\longrightarrow (0)$}.\\
\end{array}
$$
\end{lemma}
A consequence of the above exact sequences is the following crucial lemma,
a proof of which can be found in [5].  
\begin{lemma} 
\begin{enumerate}
\item Let $\tau = ws_{\alpha}$, $l(\tau) = l(w)+1$. If
$\langle \lambda , \alpha \rangle \geq 0$ then 
$H^{j}(\tau , \lambda) = H^{j}(w, V_{\lambda, \alpha})$ for all $j\geq 0$.
\item Let $\tau = ws_{\alpha}$, $l(\tau) = l(w)+1$. If
$\langle \lambda ,\alpha \rangle  \geq 0$, then $H^{i}(\tau , \lambda ) =
H^{i+1}(\tau , s_{\alpha}\cdot \lambda)$. Further, if
$\langle \lambda , \alpha \rangle  \leq -2$, then $H^{i}(\tau , \lambda ) =
H^{i-1}(\tau ,s_{\alpha}\cdot \lambda)$. 
\item If $\langle \lambda , \alpha \rangle  = -1$, then $H^{i}( \tau ,\lambda)$ vanishes for every $i\geq 0$ (cf. [7], Prop 5.2(b) ).
\end{enumerate}
\end{lemma}

We derive the following easy consequence of the lemma(2.2) which will be used 
to compute cohomologies in this paper:

\begin{lemma}
Let $V$ be an irreducible  $L_{\alpha}$- module. Let $\lambda$
be a character of $B_{\alpha}$. Then, we have 
\begin{enumerate}
\item If
$\langle \lambda , \alpha \rangle \geq 0$, then 
$H^{0}(L_{\alpha}/B_{\alpha} , V\bigotimes \mathbb{C}_{\lambda})$ 
is isomorphic to the tensor product of $V$ and $H^{0}(L_{\alpha}/B_{\alpha} , \mathbb{C}_{\lambda})$, and 
$H^{j}(L_{\alpha}/B_{\alpha} , V\bigotimes \mathbb{C}_{\lambda}) =(0)$ 
for every $j\geq 1$.

\item If
$\langle \lambda , \alpha \rangle  \leq -2$, 
$H^{0}(L_{\alpha}/B_{\alpha} , V\bigotimes \mathbb{C}_{\lambda})=(0)$, 
and $H^{1}(L_{\alpha}/B_{\alpha} , V\bigotimes \mathbb{C}_{\lambda})$, 
is isomorphic to the tensor product of  $V$ and $H^{0}(L_{\alpha}/B_{\alpha} , 
\mathbb{C}_{s_{\alpha}\cdot\lambda})$. 

\item If $\langle \lambda , \alpha \rangle  = -1$, then 
$H^{j}( L_{\alpha}/B_{\alpha} , V\bigotimes \mathbb{C}_{\lambda}) =(0)$ 
for every $j\geq 0$.
\end{enumerate}
\end{lemma}

We recall the following lemma from [2] on indecomposable $B_{\alpha}$- modules 
(cf. [2], cor(9.1) ).

\begin{lemma}
Any finite dimensional indecomposable $B_{\alpha}$ module $V$ is isomorphic to 
$V^{\prime}\bigotimes \mathbb{C}_{\lambda}$ for some irreducible representation
$V^{\prime}$ of $L_{\alpha}$, and $\mathbb{C}_{\lambda}$ is an one dimensional 
representation of $L_{\alpha}$ given by a character $\lambda$ of $B_{\alpha}$.
\end{lemma}

Applying lemma(2.4), we obtain the following lemma.

Let $V$ be a $P_{\alpha}$ module. Consider the restriction of the module $V$ 
to $B$. Consider the evaluation map 
$ev:H^{0}(P_{\alpha}/B, V)\longrightarrow V$ defined by $ev(s)=s(idB)$,
the value of $s$ at the identity coset $idB$ of $P_{\alpha}/B$.

Then, we have

\begin{lemma} 
\begin{enumerate}
\item
The evaluation map $ev: H^{0}(P_{\alpha}/B, V) \longrightarrow V$
is an isomorphism of $P_{\alpha}$- modules.
\item
$H^{i}(P_{\alpha}/B, V)=(0)$ for all $i\geq 1$.
\end{enumerate}
\end{lemma}
\begin{proof}
Since the inclusion $L_{\alpha}/B_{\alpha}\hookrightarrow P_{\alpha}/B$ is an 
isomorphism, by treating the $B$- module as a $B_{\alpha}$- module, 
it is sufficient to prove that the evaluation map $ev:H^{0}(L_{\alpha}/B_{\alpha}, V)\longrightarrow V$ is an isomorphism and    
$H^{i}(L_{\alpha}/B_{\alpha}, V)=(0)$ for all $i\geq 1$.   

Now, we decompose $V$ into irreducible $L_{\alpha}$- modules.
This is possible since $L_{\alpha}$ is reductive, and the base field is $\mathbb{C}$.

Since, the cohomologies commute with direct sum, we may assume that 
$V$ is an irreducible $L_{\alpha}$- module.

Now, the lemma follows from lemma(2.3(1)), by taking $\lambda=0$.
\end{proof}  

We state a combinatorial lemma. 
For completeness, we give a proof here.

\begin{lemma}
Let $G$ be a simple simply laced algebraic group.  
Let $\alpha\in S$, and $\beta$ be a root different from both $\alpha$ and 
$-\alpha$.
Then, $\langle \beta , \alpha \rangle \in \{-1, 0, 1\}$.
\end{lemma}
\begin{proof}
Since $\beta$ and $\alpha$ are not proportional, by using similar 
arguements in [6] we see that the product  
$\langle \beta , \alpha \rangle$ $\langle \alpha , \beta \rangle$ is an 
integer lying in $\{0, 1, 2, 3 \}$.

Since $G$ is simply laced, we have 
$\langle \beta , \alpha \rangle =\langle \alpha , \beta \rangle$.

Since $\langle \beta , \alpha \rangle$ is an integer, 
$\langle \beta , \alpha \rangle \in \{-1, 0, 1\}$.

\end{proof} 

Let $\gamma$ be a simple root.

We recall that $sl_{2, \gamma}$ is the simple Lie algebra corresponding to 
$\gamma$. 

We first note that $sl_{2, \gamma}$ is an indecomposable 
$B_{\gamma}$- summand of $\mathfrak{g}$.

The following lemma gives a description of indecomposable $B_{\gamma}$-summands
of $\mathfrak{g}$.

\begin{lemma}
Every indecomposable $B_{\gamma}$ summand $V$ of 
$\mathfrak{g}$ must be one of the following:
\begin{enumerate}
\item
$V=\mathbb{C}\cdot h$ for some $h\in \mathfrak{h}$ such that 
$\gamma(h)=0$. \\
\item  $V=\mathfrak{g}_{\beta} \bigoplus \mathfrak{g}_{\beta-\gamma}$ 
for some root $\beta$ such that 
$\langle \beta , \gamma \rangle = 1$. \\ 
\item
$V=sl_{2,\gamma}$, the three dimensional irreducible 
$L_{\gamma}$-module with highest weight $\gamma$.\\
\end{enumerate}
\end{lemma}
\begin{proof}

Let $V$ be an indecomposable $B_{\gamma}$-summand of $\mathfrak{g}$.
Let $\lambda$ be a maximal weight of $V$. Then, the direct sum $\bigoplus_{r\in \mathbb{Z}_{\geq 0}}V_{\lambda-r\gamma}$ is a $B_{\gamma}$-summand of $V$.  

Hence, we have $V=\bigoplus_{r\in \mathbb{Z}_{\geq 0}}V_{\lambda-r\gamma}$.
By lemma(2.6), the dimension of $V$ must be atmost two unless
$V=sl_{2,\gamma}$. 

Further, if the dimension of $V$ is one, $V=\mathbb{C}\cdot h$ for some 
$h\in \mathfrak{h}$ such that $\gamma(h)=0$. Also, if the dimension 
of $V$ is two, then, we must have $V=\mathfrak{g}_{\beta} \bigoplus \mathfrak{g}_{\beta-\gamma}$  for some root $\beta$ such that $\langle \beta , \gamma \rangle = 1$.  

This completes the proof of the lemma.

\end{proof}

\section{Proof of theorem A- simply laced Case}

In this section, we prove theorem A. Theorem A is stated for only simply laced 
groups. However, in the first subsection we prove a result for any 
simple algebraic group over any alberaically closed field of arbitrary 
charactristic.

\subsection{Global sections $H^{0}(G/B, V)$ for the case  when $V$ is a $G$-module}

We have the following notation only in this subsection. Let $K$ be an 
algebraically closed field of arbitrary characteristic. Let $G$ be a simple, 
simply connected algebraic group over $K$.

In this subsection, we prove that for any $G$-module $V$, the evaluation map

$ev: H^{0}(G/B, V)\longrightarrow V$ given by $ev(s)=s(idB)$ is an isomoprhism 
of $G$-modules (cf lemma(3.2)).

This lemma is a slight generalisation of lemma(2.5(1)). Also, its proof is 
independent of the characteristic of the base field.  

We first prove the following two basic lemmas. For 
the completeness , we provide a proof here.

Let $H$ be an algebraic group over $K$.

Let $W_{1}$ and $W_{2}$ be two finite dimensional rational $H$-modules.

Let $W_{2}^{H}$ denote the set of all $H$-invariants of $W_{2}$.
Let $Hom_{H}(W_{1}, W_{2})$ denote the set of all homomorphism of $H$-modules
from $W_{1}$ into $W_{2}$.

Consider the linear map $\psi:(W_{1}^{*}\bigotimes W_{2})^{H}\longrightarrow Hom_{H}(W_{1}, W_{2})$ given by 

$\psi(f\otimes w)(v)=f(v)\cdot w$.

Then, we have

\begin{lemma}
The restriction of $\psi$ to $(W_{1}^{*}\bigotimes W_{2})^{H}$
induces an isomorphism 

$\psi:(W_{1}^{*}\bigotimes W_{2})^{H} \longrightarrow 
Hom_{H}(W_{1}, W_{2})$ of finite dimensional vector spaces over $K$.
\end{lemma}
\begin{proof}

Let $\sum_{i=1}^{r}f_{i}\otimes e_{i}\in (W_{1}^{*}\bigotimes W_{2})^{H}$.
Let $\psi(\sum_{i=1}^{r}f_{i}\otimes e_{i})=\phi$.

For any $v\in W_{1}$ and for any $h\in H$, we have 
$\phi(h\cdot v)=\sum_{i=1}^{r}f_{i}(h\cdot v) \cdot e_{i}$.
Since $\sum_{i=1}^{r}f_{i}\otimes e_{i}$ is $H$-invariant, we have  
$$\sum_{i=1}^{r}f_{i}(h\cdot v) \cdot e_{i} =\sum_{i=1}^{r}f_{i}(v)\cdot ( h\cdot e_{i})=h\cdot \phi(v).$$

Hence, we have $\phi(h\cdot v)= h\cdot \phi(v)$. Thus, we can see that 
$\psi(W_{1}^{*}\bigotimes W_{2})^{H} \subset Hom_{H}(W_{1}, W_{2})$.

Proof of $Hom_{H}(W_{1}, W_{2}) \subset \psi(W_{1}^{*}\bigotimes W_{2})^{H}$
is similar to that of 

$\psi(W_{1}^{*}\bigotimes W_{2})^{H} \subset Hom_{H}(W_{1}, W_{2})$.

This completes the proof of the lemma.

\end{proof}

The following lemma could be well known. For completeness, we give the 
details of a proof. 

We use lemma(3.1)  to prove:

\begin{lemma}
Let $V$ be a finite dimensional rational $G$-module. Then, the evaluation map 
$ev: H^{0}(G/B, V)\longrightarrow V$ is an isomorphism of $G$-modules.
\end{lemma}
\begin{proof}
{\it Step 1}

We first show that the evaluation map 
$ev: H^{0}(G/B, V)\longrightarrow V$ is a homomorphism of $G$-modules.

Take $W_{1}=H^{0}(G/B, V)$ and taking $W_{2}=V$.

We first note that $W_{1}$ and $W_{2}$ are both $G$-modules.
Since $G/B$ is projective, we see that 

{\it Observation}

the $B$-invariants of  
$W_{1}^{*}\bigotimes W_{2}$ is equal to the $G$-invariants of 
$W_{1}^{*}\bigotimes W_{2}$.

Since the evaluation map $ev: H^{0}(G/B, V)\longrightarrow V$ is a
homomorphism of $B$-modules, applying lemma(3.1) to $H=B$, we  can find a vector $u \in (W_{1}^{*}\bigotimes W_{2})^{B}$ such that  
$\psi(u)= ev$.

By the above {\it Observation}, we have $u \in (W_{1}^{*}\bigotimes W_{2})^{G}$.
We now apply lemma(3.1) to $H=G$ and conclude that 
the evaluation map $ev: H^{0}(G/B, V)\longrightarrow V$ is a
homomorphism of $G$-modules.

This completes the proof of {\it Step 1}.

By the description of the global sections of the vector bundle on $G/B$ 
associated to the $B$-module,  $H^{0}(G/B, V)$ is the space of all morphisms 
$f:G \longrightarrow V$ satisfying $f(gb)=b^{-1}\cdot f(g)$.

For each $v\in V$, we associate a morphism  
$\phi_{v}::G\longrightarrow V$ defined by $\phi_{v}(g)=g^{-1}\cdot v$.
Clearly, $\phi_{v}(gb)=b^{-1}\cdot \phi_{v}(g)$ for every $g\in G$ and 
for every $b\in B$.

So, we have the map 
$\phi:V \longrightarrow H^{0}(G/B, V)$ given by $\phi(v)=\phi_{v}$.
Clearly $\phi$ is injective. Hence, the dimension of $V$  is atmost 
the dimension of $H^{0}(G/B, V)$.

On the otherhand, using {\it Step 1}, we see that the kernal of evaluation map 
$ev:H^{0}(G/B, V)\longrightarrow V$ is a $G$ submodule of  $H^{0}(G/B, V)$.
Now, let $f$ be in the kernel of $ev$. Then, $g^{-1}\cdot f$ is also in the
kernel of $ev$. Hence, we have $f(g)=0$ for every $g\in G$. Thus , 
we have $f=0$.

Hence $ev$ is injective. Since the dimension of $V$  is atmost 
the dimension of $H^{0}(G/B, V)$, the evaluation map 
$ev: H^{0}(G/B, V)\longrightarrow V$ is an isomorphism of $G$-modules.

This completes the proof the lemma.

\end{proof}  

\subsection{Proof of theorem A}

In this section, we prove theorem A.

The follwing notation will be maintained throughout the rest of this 
section.

Let $G$ be a simple, simply connected and  simply laced algebraic group
over $\mathbb{C}$. Let $\mathfrak{g}$ be the Lie algebra of $G$.

Let $\tau\in W$. Let $\gamma$ be a simple root. 
Let $V$ be a $B$-sub module of $\mathfrak{g}$ 
containing $\mathfrak{b}$.  
We recall the evaluation map $ev:H^{0}(\tau,V)\longrightarrow V$
by $ev(s)=s(idB)$, the evaluation of the section at the identity
coset $idB$ as point in $X(\tau)$.  

\begin{lemma}
The evaluation map 
$ev:H^{0}(\tau, V)\longrightarrow V$ is injective.
\end{lemma}
\begin{proof}

Proof is by induction on $l(\tau)$.

If $l(\tau)=0$, we are done.

So, we may choose a simple root $\gamma\in S$ such that 
$l(\tau)=l(s_{\gamma}\tau)=l(\tau)-1$.

Then, by induction on $l(\tau)$, we assume that the evaluation map 

$ev:H^{0}(s_{\gamma}\tau , V) \longrightarrow V$ is injective.

Since $V$ is a $B$-submodule of $\mathfrak{g}$, 
$H^{0}(s_{\gamma}\tau , V)$ is a $B$- submodule of 
$H^{0}(s_{\gamma}\tau , \mathfrak{g})$.  

Hence, $H^{0}(P_{\gamma}/B, H^{0}(s_{\gamma}\tau , V))$ is a $B$-submodule 
$H^{0}(P_{\gamma}/B, \mathfrak{g})$.

On the other hand, since the  $B$- module $\mathfrak{g}$ is a restriction 
of a $P_{\gamma}$ module, and so $\mathfrak{g}$ is a $L_{\gamma}$- module.

Now, since the incusion $L_{\gamma}/B_{\gamma}\hookrightarrow P_{\gamma}/B$
is an isomorphism, by using lemma (2.5(1)) (or lemma(3.2)), the evaluation map 
$ev:H^{0}(P_{\gamma}/B, \mathfrak{g})\longrightarrow \mathfrak{g}$ is an 
isomorphism of $B$- modules. Hence, $H^{0}(P_{\gamma}/B, H^{0}(s_{\gamma}\tau , V))$ is a $B$-submodule of $\mathfrak{g}$. 

Hence, the evaluation map $ev: H^{0}(P_{\gamma}/B, H^{0}(s_{\gamma}\tau , V))\longrightarrow H^{0}(s_{\gamma}\tau , V)$ is injective. Hence, using the short 
exact sequence {\it SES} of $B$- modules, we see that $H^{0}(\tau, V)$ 
is isomorphic to $H^{0}(P_{\gamma}/B, H^{0}(s_{\gamma}\tau , V))$.

Now, since  the evaluation map $ev:H^{0}(\tau,V)\longrightarrow V$ is the 
composition of the evaluation maps 
$ev: H^{0}(P_{\gamma}/B, H^{0}(s_{\gamma}\tau , V))
\longrightarrow H^{0}(s_{\gamma}\tau , V)$, and 
$ev:H^{0}(s_{\gamma}\tau, V)\longrightarrow V$, we conclude that
the evaluation map $ev: H^{0}(\tau, V)\longrightarrow V$ is injective.

This completes the proof of the lemma.
\end{proof}

Let $\tau\in W$. Let $\gamma$ be a simple root. 

Now, let $V$ be a $B$-sub module of $\mathfrak{g}$ 
containing $\mathfrak{b}$. Then, we have 

In view of lemma(2.7) and lemma(3.3), we see that the indecomposable 
$B_{\gamma}$- summands of $H^{0}(\tau , V)$ must be atmost $3$- dimensional.
However, it is not clear what are they precisely. It is important 
to study them to determine the cohomolgy modules $H^{i}(\tau, V)$.  

In this context, we prove the following Key lemma.

\begin{lemma}
Let $\tau\in W$. Let $\gamma$ be a simple root.
Every indecomposable $B_{\gamma}$- summand $V^{\prime}$ of 
$H^{0}(\tau , V)$ must be one of the following:
\begin{enumerate}
\item
$V^{\prime}=\mathbb{C}\cdot h$ for some $h\in \mathfrak{h}$ such that 
$\gamma(h)=0$. \\
\item  
$V^{\prime}=\mathbb{C}\cdot h \bigoplus \mathfrak{g}_{-\gamma}$ for some 
$h \in\mathfrak{h}$ such that $\gamma(h)=1$ and $\nu(h)=0$ for every 
simple root $\nu$ different from $\gamma$.\\
\item
$V^{\prime}=\mathfrak{g}_{\beta}$ for some root $\beta$ such that 
$\langle \beta , \gamma \rangle$ lying in $\{-1, 0, 1\}$. \\
\item  $V^{\prime}=\mathfrak{g}_{\beta} \bigoplus \mathfrak{g}_{\beta-\gamma}$ 
for some root $\beta$ such that 
$\langle \beta , \gamma \rangle = 1$. \\ 
\item
$V^{\prime}=sl_{2,\gamma}$, the three dimensional irreducible 
$L_{\gamma}$-module with highest weight $\gamma$.\\
\end{enumerate}
\end{lemma}
\begin{proof}

Let $V^{\prime}$ be an indecomposable $B_{\gamma}$-summand of
$H^{0}(\tau, V)$. If the weight of the $B_{\gamma}$-stable line in 
$V^{\prime}$ is different from $-\gamma$, then, using lemma(2.7) and 
lemma(3.3), we see that $V^{\prime}$ must be one of the types (1), (3) or (4).

Otherwise, $\mathfrak{g}_{-\gamma}$ is a $B_{\gamma}$-submodule of 
$V^{\prime}$. In this case, we need to show that  
$\mathfrak{g}_{-\gamma}$ is a proper subspace of $V^{\prime}$.
That is, either  $V^{\prime}=\mathbb{C}\cdot h \bigoplus \mathfrak{g}_{-\gamma}$ 
for some $h \in\mathfrak{h}$ such that $\gamma(h)=1$ and $\nu(h)=0$ for every 
simple root $\nu$ different from $\gamma$ or $V^{\prime}=sl_{2, \gamma}$.

We prove this by induction on $l(\tau)$.

If $l(\tau)=0$, then, $\tau=id$ and so we are done.

Otherwise, choose a simple root $\alpha$ such that 
$l(\tau)=1+l(s_{\alpha}\tau)$.

By induction on $l(\tau)$, we assume that for any simple root $\nu$, if $V^{\prime}$ is an indecomposable $B_{\nu}$ summand of $H^{0}(s_{\alpha}\tau, V)$ containing $\mathfrak{g}_{-\nu}$, then, either $V^{\prime}=\mathbb{C}\cdot h \bigoplus 
\mathfrak{g}_{-\nu}$ for some $h \in\mathfrak{h}$ such that $\nu(h)=1$ and $\mu(h)=0$ for every simple root $\mu$ different from $\nu$ or
$V^{\prime}=sl_{2, \nu}$.

We now fix a simple root $\gamma$.

We give the details of proof in $3$ different cases as follows.

{\it Case 1:}

We first assume that $\gamma=\alpha$. 

Now, let $V_{1}$ be an indecomposable $B_{\gamma}$-summand of  
$H^{0}(\tau, V)$ containing $\mathfrak{g}_{-\gamma}$. Then, 
using lemm(3.3), we see that there is an indecomposable $B_{\gamma}$-summand 
$V^{\prime}$ of $H^{0}(s_{\gamma}\tau, V)$ containing 
$\mathfrak{g}_{-\gamma}$. Since $l(s_{\gamma}\tau)=l(\tau)-1$, by induction on length of $\tau$, we see that $V^{\prime}$ must be of type
(2) or of type (5).

If $V^{\prime}$ is of type (2), then, $H^{0}(s_{\gamma}, V^{\prime})=(0)$.
Hence, $\mathfrak{g}_{-\gamma}$ can not be a subspace of 
$H^{0}(s_{\gamma}, V^{\prime})$.

On the other hand,  using {\it SES}, we have 
$H^{0}(s_{\gamma}, H^{0}(s_{\gamma}\tau , V))=H^{0}(\tau, V)$. 
Thus, $\mathfrak{g}_{-\gamma}$ can not be a subspace of 
$H^{0}(\tau, V)$. 

This completes the proof for the case when $\alpha=\gamma$.

{\it Case2 :}

We assume that $\alpha$ is different from $\gamma$ and 
$\langle \gamma , \alpha \rangle \neq 0$.
By using lemma (2.6), we have $\langle \gamma , \alpha \rangle = -1$.

By a similar arguement as in {\it Case 1}, we may assume that there is a

$B_{\gamma}$-summand 
$V^{\prime}$ of $H^{0}(s_{\gamma}\tau, V)$ containing 
$\mathfrak{g}_{-\gamma}$. By induction, $V^{\prime}$ must be of type
(2) or of type (5).

Now, let $U^{\prime}$ denote the minimal $B_{\alpha}$-summand of  
$H^{0}(\tau, V)$ containing $V^{\prime}$. 

{\it Subcase 1:}

If $V^{\prime}$ is of type (2), then, $V^{\prime}=\mathbb{C}\cdot h \bigoplus 
\mathfrak{g}_{-\gamma}$ for some $h \in\mathfrak{h}$ such that 
$\gamma(h)= 1$ and $\nu(H)=0$ for every simple root $\nu$ different from 
$\gamma$.

Then, an indecomposable $B_{\alpha}$-summand $V_{1}$ of $U^{\prime}$ containing 
$\mathfrak{g}_{-\gamma}$ must be of the form $\mathfrak{g}_{-\gamma} \bigoplus \mathfrak{g}_{-\gamma-\alpha}$. So, by lemma (2.5), we have 
$H^{0}(s_{\alpha}, V_{1})=V_{1}$.

Hence, $\mathfrak{g}_{-\gamma}$ must be a sub space of 
$H^{0}(s_{\gamma}, H^{0}(s_{\gamma}\tau, V))$.

Since $\alpha\neq \gamma$, we have $\alpha(h)=0$. Hence,  
$\mathbb{C}\cdot h $ is a $B_{\alpha}$-direct summand 
of $U^{\prime}$. Hence, $\mathbb{C}\cdot h $ must be a $B_{\alpha}$-submodule of 
$H^{0}(s_{\alpha} , U^{\prime})$. Hence, $\mathbb{C}\cdot h$ must be a subspace 
of $H^{0}(s_{\gamma}, H^{0}(s_{\gamma}\tau, V))$.

Hence, $V^{\prime}=\mathbb{C}\cdot h \bigoplus 
\mathfrak{g}_{-\gamma}$ must be a subspace of 
$H^{0}(s_{\alpha}, H^{0}(s_{\alpha}\tau, V))$. 

Thus, by using {\it SES}, we conclude that 
$V^{\prime}=\mathbb{C}\cdot h \bigoplus 
\mathfrak{g}_{-\gamma}$ is a subspace of 
$H^{0}(\tau, V))$. 

{\it Subcase 2:}

Let $V^{\prime}$ be of type (5). Then, we have $V^{\prime}=sl_{2, \gamma}$.
In this case, $U^{\prime}$, the minimal $B_{\alpha}$-summand of  
$H^{0}(\tau, V)$ containing $V^{\prime}$ must contain 
$H_{\gamma}=-[X_{-\gamma}, X_{\gamma}]$. 
Here $[X_{-\gamma}, X_{\gamma}]$ denotes the Lie bracket of  $X_{-\gamma}$ and 
$X_{\gamma}$ in $\mathfrak{g}$.

Since $\alpha(H_{\gamma})=-1$, we have $[X_{-\alpha}, H_{\gamma}]=X_{-\alpha}$. 
Hence, $\mathfrak{g}_{-\alpha}$ must be a subspace of $U^{\prime}$. 

Therefore, by induction applying to the simple root $\nu=\alpha$, 
$U^{\prime}$ must either contain $sl_{2, \alpha}$ or 
the indecomposable $B_{\alpha}$- module $\mathbb{C}\cdot h \bigoplus \mathfrak{g}_{-\alpha}$ for some $ h\in\mathfrak{h}$ such that $\alpha(h)=1$ and 
$\nu(h)=0$ for every simple root $\nu$ different from $\alpha$.

Now, if $\mathfrak{g}_{\gamma+\alpha}$ is a sub space of $U^{\prime}$, then, 
$\mathfrak{g}_{\alpha+\gamma}\bigoplus \mathfrak{g}_{\gamma}$ is an 
indecomposable $B_{\alpha}$- summand of $U^{\prime}$. Since $\mathfrak{g}_{\alpha+\gamma}+\mathfrak{g}_{\gamma}$ is a $L_{\gamma}$- module, using lemma(2.5), 
we see that $H^{0}(s_{\alpha}, \mathfrak{g}_{\alpha+\gamma}\bigoplus \mathfrak{g}_{\gamma})=\mathfrak{g}_{\alpha+\gamma}\bigoplus\mathfrak{g}_{\gamma}$. Hence, 
we have 

{\it Observation:} 

$\mathfrak{g}_{\alpha+\gamma}\bigoplus\mathfrak{g}_{\gamma}$ is a sub space of 
$H^{0}(s_{\alpha}, U^{\prime})$.  

Since $U^{\prime}$ is a $B_{\alpha}$-submodule of 
$H^{0}(s_{\alpha}\tau, V)$, $H^{0}(s_{\alpha}, U^{\prime})$ is a $B_{\alpha}$-sub
module of $H^{0}(s_{\alpha}, H^{0}(s_{\alpha}\tau, V))$. On the other hand, 
since $H^{0}(s_{\alpha}, H^{0}(s_{\alpha}\tau, V))$
is a $B_{\gamma}$-module , using {\it Observation } , we see that 
the $B_{\gamma}$-span $sl_{2,\gamma}$ of $\mathfrak{g}_{\gamma}$ 
must be a $B_{\gamma}$-sub module of $H^{0}(s_{\alpha}, H^{0}(s_{\alpha}\tau, V))$.

Using {\it SES}, we conclude that $H^{0}(\tau, V)$ cntains 
$sl_{2,\gamma}$. This proves that $H^{0}(\tau, V)$ cntains an indecomposable 
$B_{\gamma}$ summand of type (5).

Now, if $\mathfrak{g}_{\gamma+\alpha}$ is not a sub space of $U^{\prime}$, then, 
$\mathfrak{\gamma}$ is an indecomposable $B_{\alpha}$- direct summand of $U^{\prime}$. Since $\langle \gamma , \alpha \rangle =-1$, by lemma , we have
$H^{i}(s_{\alpha} ,  \mathfrak{\gamma}) =(0)$ for every $i \in \mathbb{Z}_{\geq 0}$. In particular, $\mathfrak{\gamma}$ can not be a subspace of $H^{0}(s_{\alpha}, H^{0}(s_{\alpha}\tau, V))$. 

Let $V^{\prime}$ be of type (5). Then, we have $V^{\prime}=sl_{2, \gamma}$.
In this case, $U^{\prime}$, the minimal $B_{\alpha}$-summand of  
$H^{0}(\tau, V)$ containing $V^{\prime}$ must contain 
$B_{\gamma}$-span of $X_{\gamma}$. In particular, 
$H_{\gamma}\in U^{\prime}$. Since $\alpha(H_{\gamma})=-1$, we have $[X_{-\alpha}, H_{\gamma}]=X_{-\alpha}$. Hence, $\mathfrak{g}_{-\alpha}$ must be a sub space of $U^{\prime}$.

Therefore, by induction applying to the simple root $\alpha$, $U^{\prime}$
must either contain $sl_{2, \alpha}$ or it must contain the indecomposable 
$B_{\alpha}$- module $\mathbb{C}\cdot h \bigoplus \mathfrak{g}_{-\alpha}$ for 
some $h \in\mathfrak{h}$ such that $\alpha(h)=1$ and $\nu(h)=0$ for every 
simple root $\nu$ different from $\alpha$. In either cases, $U^{\prime}$ 
contains a vector $h^{\prime}$ of $\mathfrak{h}$ which is linearly independent
to $H_{\gamma}$ and $\alpha(h^{\prime})=1$. 

Hence, we can find a vector $h$ in the vector subspace spanned by $h^{\prime}$ 
and $H_{\gamma}$ such that $\gamma(h)=1$ and $\nu(h)=0$ for every simple root 
different from $\gamma$. Therefore, $\mathbb{C}\cdot h$ is a $B_{\alpha}$- direct summand of $U^{\prime}$. Hence, we see that $H^{0}(s_{\alpha}, \mathbb{C}\cdot h)=\mathbb{C}\cdot h$ is a subspace of $H^{0}(s_{\alpha}, U^{\prime})$.

Thus, the $B_{\gamma}$-span $\mathbb{C}\cdot h \bigoplus 
\mathfrak{g}_{-\gamma}$ of $\mathbb{C}\cdot h$ is a subspace of 
$H^{0}(s_{\alpha}, U^{\prime})$. Using {\it SES}, we conclude that 
$\mathbb{C}\cdot h \bigoplus \mathfrak{g}_{-\gamma}$ is a $B_{\gamma}$- direct 
summand of $H^{0}(\tau, V))$. 

{\it Case 3:}

We assume that $\langle \gamma , \alpha \rangle =0$. 

Proof in this case is similar but actually simpler than that of 
{\it Case 2}.

\end{proof}

Let $G$ be simply laced. 

Let $V$ be a $B$-sub module of $\mathfrak{g}$ containing $\mathfrak{b}$.

\begin{lemma}
Let $\tau \in W$. Then, we have
$H^{i}(\tau, V)=(0)$ for every $i\geq 1$.
\end{lemma}
\begin{proof}

Proof is by induction on $l(\tau)$.

If $l(\tau)=0$, we are done.

Otherwise, we choose a simple root $\gamma\in S$ be such that 
$l(s_{\gamma}\tau)=l(\tau)-1$.

By lemma(3.4), every indecomposable $B_{\gamma}$- summand $V^{\prime}$ 
of $H^{0}(s_{\gamma}\tau , V)$ must be one of the 5 types given 
in lemma(3.4).    

Hence, using lemma(2.3), we conclude that 
$H^{i}(P_{\gamma}/B, V^{\prime})$ is zero for every indecomposable 
$B_{\gamma}$-summand $V^{\prime}$ of  $H^{0}(s_{\gamma}\tau , V))$ 
and for every $i\geq 1$.

Thus, we have shown that 

{\it Observation :}

$H^{i}(P_{\gamma}/B, H^{0}(s_{\gamma}\tau , V))=(0)$ 
for all $i\geq 1$.

By induction on $l(\tau)$, we have $H^{i}(s_{\gamma}\tau , V)$ is zero 
for all $i\geq 1$. Now, using {\it Observation } and using the short exact 
sequence {\it SES} of $B$ modules, we conclude that 
$H^{i}(\tau, V)$ is zero for all $i\geq 1$.

This completes the proof of lemma.

\end{proof}

Let $V_{1}$ be a $B$-sub module of $\mathfrak{g}$ containing $\mathfrak{b}$.
Let $V_{2}$ be a $B$-sub module of $V_{1}$ containing $\mathfrak{b}$.

Let $\tau\in W$. The natural projection $\Pi:V_{1} \longrightarrow 
V_{1}/V_{2}$ of $B$-modules induces a homomorphism of $B$-modules 
$\Pi_{\tau}:H^{0}(\tau, V_{1})\longrightarrow H^{0}(\tau, V_{1}/V_{2})$ of 
$B$-modules.

We now deduce the following lemma as a consequence of the lemma(3.5).

\begin{lemma}
\begin{enumerate}
\item
$H^{i}(\tau, V_{1}/V_{2})$ is zero for all $i\geq 1$. \\
\item
$\Pi_{\tau}:H^{0}(\tau, V_{1}) \longrightarrow H^{0}(\tau, V_{1}/V_{2})$ is a 
surjective homomorphism of $B$-modules whose kernel is 
$H^{0}(\tau ,V_{2})$. \\
\end{enumerate}
\end{lemma}
\begin{proof}

{\it Proof of (1):}

We have the short exact sequence    
 $(0)\longrightarrow V_{2} \longrightarrow V_{1} \longrightarrow 
V_{1}/V_{2} \longrightarrow (0)$ of $B$- modules.

Applying $H^{i}(\tau,- )$ to this short exact sequence of $B$-modules, 
we obtain the following long exact sequence of $B$-modules: 

{\it Observation :}

$$\cdots H^{i}(\tau, V_{2}) \longrightarrow H^{i}(\tau, V_{1}) 
\longrightarrow H^{i}(\tau, V_{1}/V_{2})\longrightarrow 
H^{i+1}(\tau, V_{2}) \cdots$$

By lemma(3.5), $H^{i}(\tau, V_{2})$, $H^{i}(\tau, V_{1})$ and  
$H^{i+1}(\tau, V_{2})$ are all zero for every $i\geq 1$.
Thus, we conclude that $H^{i}(\tau, V_{1}/V_{2})=(0)$
for every $i\geq 1$. 

This proves (1).

{\it Proof of (2):}

Taking $i=0$ in {\it Observation } and using $H^{1}(\tau, V_{2})=(0)$,  
we obtain the following short exact sequence 
$$(0)\longrightarrow H^{0}(\tau, V_{2}) \longrightarrow H^{0}(\tau, V_{1}) 
\longrightarrow H^{0}(\tau, V_{1}/V_{2})\longrightarrow (0).$$

This proves (2).

\end{proof}

We have 

\begin{corollary}
Let $\tau\in W$. Let $\alpha$ be a positive root.
Then, $H^{i}(\tau, \alpha)=(0)$ for every $i\geq 1$.
\end{corollary}
\begin{proof}
 Let $V_{1}:=\bigoplus_{\mu\leq \alpha}\mathfrak{g}_{\mu}$ denote the 
direct sum of the weight spaces of $\mathfrak{g}$ of weights $\mu$ satisfying
$\mu\leq \alpha$. 

Let $V_{2}:=\bigoplus_{\mu < \alpha}\mathfrak{g}_{\mu}$ denote the 
direct sum of the weight spaces of $\mathfrak{g}$ of weights $\mu$ satisfying
$\mu < \alpha$. 

It is clear that $V_{2}$ is a $B$-sub module of $\mathfrak{g}$ containing $\mathfrak{b}$ and $V_{1}$ is a $B$-sub module of $\mathfrak{g}$ containing $V_{2}$.

Since $\mathfrak{g}_{\alpha}$ is one dimensional and is isomorphic to
the quotient $V_{1}/V_{2}$, we have 

$H^{i}(\tau, \alpha)= H^{i}(\tau, \mathfrak{g}_{\alpha})=
H^{i}(\tau, V_{1}/V_{2})$ for every $i\geq 1$.
Hence, by lemma(3.6), $H^{i}(\tau, \alpha)=(0)$ for every 
$i\geq 1$.

This completes the proof of corollary.
\end{proof}

We now prove the following theorem. 

Let $\tau\in W$. Let $\alpha_{0}$ denote the highest root.

Then, we have

\begin{theorem}
\begin{enumerate}
Let $G$ be simple, simply connected and simply laced algebraic group over 
$\mathbb{C}$. 
Let $\tau \in W$.
\item
$H^{i}(X(\tau), \mathcal{T}_{G/B})=(0)$ for every 
$i\geq 1$. 
\item
$H^{0}(X(\tau) , \mathcal{T}_{G/B})$ is the adjoint representation 
$\mathfrak{g}$ of  $G$ if and only if the set of semi-stable points $X(\tau^{-1})_{T}^{ss}(\mathcal{L}_{\alpha_{0}})$ is non-empty. 
\end{enumerate}
\end{theorem}
\begin{proof}

Since the tangent space of $G/B$ at the point $idB$ is 
$\mathfrak{g}/\mathfrak{b}$, the tangent bundle $\mathcal{T}_{G/B}$ 
is the homogeneous vector bundle $\mathcal{L}(\mathfrak{g}/\mathfrak{b})$ 
on $G/B$ associated to the $B$-module $\mathfrak{g}/\mathfrak{b}$.

Hence, it is sufficient to prove the following:
\begin{enumerate}
\item
$H^{i}(\tau, \mathfrak{g}/\mathfrak{b})=(0)$ for every $i\geq 1$. 
\item
$H^{0}(\tau , \mathfrak{g}/\mathfrak{b})$ is the adjoint representation 
$\mathfrak{g}$ of $G$ if and only if 
the set of semi stable points $X(\tau^{-1})_{T}^{ss}(\mathcal{L}_{\alpha_{0}})$ 
is non-empty. 
\end{enumerate}

We prove this now.

Let $V_{1}:=\mathfrak{g}$ and let $V_{2}:=\mathfrak{b}$.
The natural projection $\Pi:\mathfrak{g} \longrightarrow 
\mathfrak{g}/\mathfrak{b}$ of $B$-modules induces a 
homomorphism $\Pi_{\tau}:H^{0}(\tau, \mathfrak{g})
\longrightarrow H^{0}(\tau, \mathfrak{g}/\mathfrak{b})$ of $B$-modules.

Proof of (1) follows from lemma(3.6(1)).

Since the evaluation map $ev:H^{0}(\tau, \mathfrak{g})\longrightarrow 
\mathfrak{g}$ is an isomorphism, in order to prove (2),
it is sufficient to prove that the kerenel of the linear 
map $\Pi_{\tau}:H^{0}(\tau, \mathfrak{g}) \longrightarrow 
H^{0}(\tau, \mathfrak{g}/\mathfrak{b})$ is zero if and only if 
$X(\tau^{-1})_{T}^{ss}(\mathcal{L}_{\alpha_{0}})$ is non-empty.

We now show that the kerenel of the linear map 
$\Pi_{\tau}:H^{0}(\tau, \mathfrak{g}) \longrightarrow 
H^{0}(\tau, \mathfrak{g}/\mathfrak{b})$ is zero if and only 
if $X(\tau^{-1})_{T}^{ss}(\mathcal{L}_{\alpha_{0}})$ is non-empty.

By ([10], lemma(2.1)), $X(\tau^{-1})_{T}^{ss}(\mathcal{L}_{\alpha_{0}})$ is 
non empty if and only if $\tau^{-1}(-\alpha_{0})\in R^{+}$.

Hence, we have 

{\it Observation 1:}
$X(\tau^{-1})_{T}^{ss}(\mathcal{L}_{\alpha_{0}})$ is non empty
if and only if $-\alpha_{0}\in \tau(R^{+})$.

On the otherhand from lemma(3.6(2)), we have 
$Ker(\Pi_{\tau})=H^{0}(\tau, \mathfrak{b})$. Hence, using lemma(3.3), 
we see that $Ker(\Pi_{\tau})$ is a $B$-submodule of $\mathfrak{b}$. 

Since there is a unique $B$- stable line in $\mathfrak{b}$ and that is of 
weight $-\alpha_{0}$, we conclude that $Ker(\Pi_{\tau})$ is a non-zero 
$B$-submodule of $\mathfrak{b}$ if and only if the $-\alpha_{0}$-weight 
space of  $H^{0}(\tau, \mathfrak{b})$ is non zero.

Hence, $Ker(\Pi_{\tau})$ is non-zero if and only if 
$-\alpha_{0}\in \tau(R^{-})$.

Reformulating this statement, we have:

$Ker(\Pi_{\tau})$ is zero if and only if 
$-\alpha_{0}\in \tau(R^{+})$.

Using {\it Observation 1}, we see that 
$Ker(\Pi_{\tau})$ is zero if and only 
if the set of semi-stable points 
$X(\tau^{-1})_{T}^{ss}(\mathcal{L}_{\alpha_{0}})$ is non-empty.

This completes the proof of (2).

\end{proof}

Let $\tau\in W$. Let $\alpha_{0}$ denote the highest root.

Let $h^{0}(\tau , \alpha)$ denote the character of the $T$-module 
$H^{0}(\tau, \alpha)$. 

\begin{corollary}
$\sum_{\alpha\in R^{+}}h^{0}(\tau,\alpha)=Char(\mathfrak{g})$ if and only 
if the set of semi-stable points 
$X(\tau^{-1})_{T}^{ss}(\mathcal{L}_{\alpha_{0}})$ is non-empty.
\end{corollary}
\begin{proof}
Let $U^{+}$ denote the unipotent radical of $B^{+}$.
Let $\mathfrak{u}^{+}$ denote the Lie algebra of $U^{+}$. 

Since the natural map $\mathfrak{u}^{+}\longrightarrow \mathfrak{g}/\mathfrak{b}$ is an isomorphism, there is a total ordering 
$\{\beta_{1}, \beta_{2}, \cdots \beta_{N}\}$ of positive roots $R^{+}$ 
such that the $B$-module $\mathfrak{g}/\mathfrak{b})$ has a filtration of 
sub modules $V_{0}:=\mathfrak{g}/\mathfrak{b}\supset V_{1}\supset V_{2}\supset 
\cdots V_{N-1}\supset V_{N}=(0)$ with each successive quotients 
$V_{i}/V_{i+1}$ is one dimensional and is isomorphic to 
$\mathfrak{g}_{\beta_{i}}$.
Hence, we have $H^{j}(\tau, V_{i}/V_{i+1})=H^{j}(\tau, \beta_{i})$.

Using corollary(3.4), we have $H^{j}(\tau, V_{i}/V_{i+1})=(0)$ for every 
$j\geq 1$ and for every $i=1,2, \cdots N-1$. 

Hence, $H^{0}(\tau, \mathfrak{g}/\mathfrak{b})$ has a filtration of $B$-sub 
modules  $H^{0}(\tau, \mathfrak{g}/\mathfrak{b})\supset 
H^{0}(\tau,V_{1})\supset H^{0}(\tau,V_{2})\supset \cdots H^{0}(\tau,V_{N-1})\supset (0)$ such that each successive quotient 
$H^{0}(\tau, V_{i})/H^{0}(\tau, V_{i+1})$ is isomorphic to 
$H^{0}(\tau, \beta_{i})$.

Hence, we have 

{\it Observation 2}
$Char(H^{0}(\tau, \mathfrak{g}/\mathfrak{b}))=\sum_{i=1}^{N}Char(H^{0}(\tau, \beta_{i}))= \sum_{\alpha\in R^{+}}Char(H^{0}(\tau,\alpha))$.

Using lemma(2.5(1)), we have $H^{0}(\tau, \mathfrak{g})=\mathfrak{g}$
when ever $l(\tau)=1$. Now, using induction on $l(\tau)$, and again 
lemma(2.5(1)) successively, we conclude that 
$H^{0}(\tau, \mathfrak{g})=\mathfrak{g}$ for every $\tau\in W$.

Therefore, we have
$Char(H^{0}(\tau, \mathfrak{g}))=Char(\mathfrak{g})$. Hence, using 
{\it Step 1}, we see that 
$Char(\mathfrak{g})=Char(H^{0}(\tau, \mathfrak{g}/\mathfrak{b}))$ 
if and only if $X(\tau^{-1})_{T}^{ss}(\mathcal{L}_{\alpha_{0}})$ is non-empty.

Proof of theorem follows using {\it Observation 2} in the above statement.

\end{proof}

\section{Schubert varieties related to maximal parabolic subgroups:}

In this section, we apply the main theorem to certain Schubert varieties 
related to maximal parabolic subgroups of $G$. For a precise statement, 
see theorem(4.2).

\begin{lemma}
Let $w\in W$, and let $\beta$ be a positive root. Let $\gamma$ be 
a simple root $\gamma$ such that $l(ws_{\gamma})=l(w)-1$ and 
$\langle \beta, \gamma \rangle = -1$. Then, we have $H^{i}(w, \beta)=0$ 
for every $i\geq 0$. 
\end{lemma}
\begin{proof}

Proof of this lemma follows from lemma(2.2(3)).

\end{proof}

Let $\alpha\in S$. Let $Q_{\alpha}$ denote the maximal parabolic subgroup of $G$
containing $B$ all $s_{\beta}$, where $\beta$ running over all simple roots
different from $\alpha$. 

Let $w_{\alpha}$ denote the unique minimal representative of the longest 
element $w_{0}$ of $W$ with respect to the maximal parabolic subgroup
$Q_{\alpha}$.

Recall that $R^{+}(\tau):=\{\beta \in R^{+}:\tau(\beta)\in -R^{+}\}$.

Let $\tau\in W$ be such that $\tau\geq w_{\alpha}$.
The following theorem describe the character of $\mathfrak{g}$ in terms of 
the sum of characters $h^{0}(\tau, \beta)$ of $H^{0}(\tau,\beta)$, $\beta$ 
running over all elements of $R^{+}(\tau)$.

\begin{theorem} For any $\tau\geq w_{\alpha}$, we have 
$\sum_{\beta\in R^{+}(\tau)}h^{0}(\tau,\beta))
=Char(\mathfrak{g})$
\end{theorem}
\begin{proof}
{\it Step 1:}

We first show that $\sum_{\beta\in R^{+}}Char(H^{0}(\tau,\beta))
=Char(\mathfrak{g})$.

Since $w_{\alpha}(\omega_{\alpha})=w_{0}(\omega_{\alpha})$, it must be a 
non trivial negative dominant
character of $T$.

Since $\alpha_{0}\geq \nu$ for every simple root $\nu$,
$\langle w_{\alpha} (\omega_{\alpha}), \alpha_{0} \rangle\leq -1$.
Since the form $\langle , \rangle$ is $W$-invariant, 
$\langle \omega_{\alpha} , w_{\alpha}^{-1}(\alpha_{0}) \rangle\leq -1$.
Hence, $w_{\alpha}^{-1}(\alpha_{0})$ must be a negative root.

Using ([10], lemma(2.1)), we conclude that 
$X(w_{\alpha}^{-1})_{T}^{ss}(\mathcal{L}_{\alpha_{0}})$ is non empty.

By theorem (3.8), we conclude that 
$\sum_{\beta\in R^{+}(\tau)}Char(H^{0}(\tau,\beta))
=Char(\mathfrak{g})$.

This proves {\it Step 1}.

{\it Step 2:}

We now show that $H^{0}(\tau, \beta)=(0)$
for every $\beta\notin R^{+}(\tau)$.

We first note that a $\beta\in R^{+}$ belongs to $R^{+}(w_{\alpha})$
if and only if $\alpha \leq \beta$.

So, the highest root $\alpha_{0}$ lies in the set $R^{+}(w_{\alpha})$.

Proof of {\it Step 2} is by descending induction on $l(\tau)$.

Since $\tau\geq w_{\alpha}$ and since 
$R^{+}(w_{\alpha})=\{\nu \in R^{+}:\nu \geq \alpha\}$, we have 
$R^{+}(w_{\alpha})\subset R^{+}(\tau)$.

Now, since $\beta\notin R^{+}(\tau)$, we have $\beta \ngeq \alpha$
So, $\beta$ must be different from $\alpha_{0}$. Hence, there is 
a simple root $\gamma$ such that $\langle \beta, \gamma \rangle = -1$. 

If  $l(\tau s_{\gamma})=l(\tau)-1$, we have  
$\langle \beta, \gamma \rangle = -1$. 
By using lemma(4.1), we see that 
$H^{0}(\tau , \beta)=(0)$. 

Otherwise, we have $l(\tau s_{\gamma})=l(\tau)+1$.
So, we have $\tau s_{\gamma} \geq w_{\alpha}$ and 
$l(w_{0})-l(\tau s_{\gamma})=l(w_{0})-l(\tau)-1$.
Now, since $s_{\gamma}(\beta)\notin R^{+}(\tau s_{\gamma})$, 
using induction on $l(w_{0})-l(\tau)$, we have 

{\it Observation :}

$H^{0}(\tau s_{\gamma}, s_{\gamma}(\beta))=(0)$.

We now consider the following short exact sequence of $B$-modules:

$$(0)\longrightarrow \mathbb{C}_{\beta}\longrightarrow 
H^{0}(s_{\gamma},s_{\gamma}(\beta))\longrightarrow \mathbb{C}_{s_{\gamma}(\beta)}
\longrightarrow (0).$$

Applying $H^{0}(\tau, )$ to this short exact sequence and using 
corollary (3.4), we obtain the following short exact sequence of $B$-modules:

$$(0)\longrightarrow H^{0}(\tau, \beta)\longrightarrow 
H^{0}(\tau , H^{0}(s_{\gamma},s_{\gamma}(\beta)))\longrightarrow 
H^{0}(\tau, s_{\gamma}(\beta)) \longrightarrow (0).$$

Since $H^{0}(\tau, H^{0}(s_{\gamma},s_{\gamma}(\beta)))=
H^{0}(\tau s_{\gamma}, s_{\gamma}(\beta))$, the above short exact 
sequence can be written as:

$$(0)\longrightarrow H^{0}(\tau, \beta)\longrightarrow 
H^{0}(\tau s_{\gamma}, s_{\gamma}(\beta)))\longrightarrow 
H^{0}(\tau, s_{\gamma}(\beta)) \longrightarrow (0).$$

Now, from {\it Observation }, we have 
$H^{0}(\tau s_{\gamma}, s_{\gamma}(\beta))=(0)$.
Using this in the short exact sequence,
we conclude that $H^{0}(\tau , \beta)=(0)$.

This proves {\it Step2}.

Proof of theorem follows from {\it Step 1} and {\it Step2}.

\end{proof}

Let $G$ be a simple, simply connected and simply lacecd algebraic group over 
$\mathbb{C}$. Let $\alpha$ be a simple root.

Let $Q_{\alpha}$ be the maximal parabolic sub group of $G$ containing $B$ 
and all $s_{\beta}$, $\beta$ running over all simple roots different from 
$\alpha$.

We derive the following corollary as an application of theorem(4.2).

\begin{corollary}
Let $\mathcal{T}_{G/Q_{\alpha}}$ denote the tangent bundle of 
$G/Q_{\alpha}$. Then, we have
\begin{enumerate}
\item
$H^{i}(G/Q_{\alpha}, \mathcal{T}_{G/Q_{\alpha}})=(0)$ for every 
$i\geq 1$. 
\item
$H^{0}(G/Q_{\alpha} , \mathcal{T}_{G/Q_{\alpha}})$ is the adjoint representation 
$\mathfrak{g}$ of  $G$. 
\end{enumerate}
\end{corollary}
\begin{proof}

Let $w_{\alpha}$ denote the unique minimal representative of the longest 
element $w_{0}$ of $W$ with respect to the maximal parabolic subgroup
$Q_{\alpha}$.

Let $\phi$ denote the birational morphism from $X(w_{\alpha})$ onto 
$G/Q_{\alpha}$ given by the composition of the natural projection 
$p:G/B \longrightarrow G/Q_{\alpha}$ and the inclusion 
$X(w_{\alpha})\hookrightarrow G/B$.

Since the direct image $\phi_{*}(\mathcal{O}_{X(w_{\alpha})})$
of the structure sheaf $\mathcal{O}_{X(w_{\alpha})}$ of $X(w_{\alpha})$ is the 
structure sheaf $\mathcal{O}_{G/Q_{\alpha}}$ of $G/Q_{\alpha}$ and all higher 
direct images are zero,  we see that 
$\phi^{*}:H^{i}(G/Q_{\alpha}, \mathcal{T}_{G/Q_{\alpha}})\longrightarrow  
H^{i}(X(w_{\alpha}), \phi^{*}(\mathcal{T}_{G/Q_{\alpha}}))$ is an isomorphism 
for every $i\geq 0$. 

Let $\mathfrak{q}_{\alpha}$ denote the Lie algebra of $Q_{\alpha_{a}}$.
Then,  $p^{*}(\mathcal{T}_{G/Q_{\alpha}})$ is actually the
homogeneous vector bundle on $G/B$ associated to the $B$-module 
$\mathfrak{g}/\mathfrak{q}_{\alpha})$.

Let $Q^{+}_{\alpha}$ be the parabolic subgroup of $G$ opposite to $Q_{\alpha}$ 
containing $B^{+}$ and all $s_{\beta}$, $\beta$ running over all simple roots 
different from $\alpha$.

Since the Lie algebra of unipotent radical of $Q^{+}_{\alpha}$ is $\bigoplus_{\beta\in R^{+}(w_{\alpha})}\mathfrak{g}_{\beta}$, we can use arguments similar to proof of corollary(3.9) to obtain the following:

There is a total ordering $\{\beta_{1}, \beta_{2}, \cdots \beta_{m}\}$ of 
positive roots in $R^{+}(w_{\alpha})$ such that the $B$-module 
$\mathfrak{g}/\mathfrak{q}_{\alpha})$ has a filtration of 
sub modules $V_{0}:=\mathfrak{g}/\mathfrak{q}_{\alpha}\supset V_{1}\supset V_{2}
\supset \cdots V_{m-1}\supset V_{m}=(0)$ with each successive quotients $V_{i}/V_{i+1}$ is one dimensional and is isomorphic to $\mathfrak{g}_{\beta_{i}}$.

Hence, we have $H^{j}(\tau, V_{i}/V_{i+1})=H^{j}(\tau, \beta_{i})$.

Using corollary(3.4), we have $H^{j}(\tau, V_{i}/V_{i+1})=(0)$ for every 
$j\geq 1$ and for every $i=1,2, \cdots N-1$. 

This completes proof of (1).

Proof of (2) follows from theorem(4.2).

\end{proof}.

\section{Top cohomology  module  $H^{l(\tau)}(\tau, \tau^{-1}\cdot 0)$ }

Throughout this section, we assume that $G$ is a simple, simply connected 
and simply laced algebraic group over $\mathbb{C}$.

In this section, we show that for any $\tau\in W$
the top cohomolgy $H^{l(\tau)}(\tau, \tau^{-1}\cdot 0)$ is the 
one dimensional trivial representation of $B$. We also prove that 
for a given Coxeter element $c$ of $W$, $H^{i}(c, c^{-1}\cdot 0)$ is zero
for every $i\neq l(c)$ if and only if both 
$X(c)_{T}^{ss}(\mathcal{L}_{\alpha_{0}})$ and 
$X(c^{-1})_{T}^{ss}(\mathcal{L}_{\alpha_{0}})$ are non-empty.

We first obtain some application of theorem(4.2) in the following subsection.

\subsection{Yang-Zelevisky's proposition on Coxeter elements}

In this subsection, we obtain a corollary on Schubert varieties $X(c^{j})$ 
corresponding to some power $c^{j}$ of any given Coxeter element $c$ as an 
application of theorem(4.2). In the proof of this corollary, we use a 
Proposition about Coxeter elements by Yang and Zelevinsky. See [12, Proposition(1.3)].

We first recall that an element $c$ of $W$ is said to be a Coxeter element
if it has a reduced expression of the form 
$c=s_{\alpha_{i_{1}}}s_{\alpha_{i_{2}}}\cdots s_{\alpha_{i_{l}}}$, where $i_{j}\neq i_{k}$
when ever $j\neq k$ and $l$ is the rank of $G$.  

We now state the following proposition from [12].

\begin{proposition}[Yang-Zelevinsky]
Let $c$ be a Coxeter element. Let $\alpha$ be a simple root.
Then, there is a $j\in \mathbb{N}$ such that 
$c^{j}(\omega_{\alpha})=w_{0}(\omega_{\alpha})$. 
\end{proposition}

We now use this proposition and theorem(4.2) to obtain the following corollary. 

Let $c$ be a Coxeter element. 

\begin{corollary}
Then, there is a  
$j\in \mathbb{N}$ such that $\sum_{\beta\in R^{+}(c^{j})}h^{0}(c^{j},\beta))
=Char(\mathfrak{g})$.
\end{corollary}
\begin{proof}
We first fix a simple root $\alpha$.
By propostion(5.1), there is a  $j\in \mathbb{N}$ such that 
$c^{j}(\omega_{\alpha})=w_{0}(\omega_{\alpha})$. 
For this choice of $j$, we have 
$c^{j}\geq w_{\alpha}$.
Proof of corollary follows from theorem(4.2)
by taking $\tau=c^{j}$. 

\end{proof}

We found some interesting facts about Coxeter elements in the study of torus 
quotients. For instance, see [[10], theorem(4.2)] and see [[11], theorem(3.3)].    
 
We now obtain the following corollary from [[10], theorem(4.2)]
in the context of the character of $\mathfrak{g}$.

Let $c$ be a Coxeter element.  
Let $h$ denote the order of the Coxeter element $c$.  Let $C$ denote the 
cyclic subgroup of $W$ generated by $c$. Let $C^{\prime}$ denote the 
complement subset of the singleton set $\{id\}$ in $C$. That is, let 
$C^{\prime}:=\{c^{j}: j=1, 2, \cdots h-1 \}$.

Then, we have 

\begin{corollary}
$\sum_{\tau\in C^{\prime}}\sum_{\beta\in R^{+}(\tau)}h^{0}(\tau,\beta)
=(h-1) Char(\mathfrak{g})$ if and only if both $X(c)_{T}^{ss}(\mathcal{L}_{\alpha_{0}})$ and $X(c^{-1})_{T}^{ss}(\mathcal{L}_{\alpha_{0}})$ are non-empty.
\end{corollary}
\begin{proof}

{\it Proof of Necessary condition:}

We first prove that if $\sum_{\tau\in C^{\prime}}\sum_{\beta\in R^{+}(\tau)}h^{0}(\tau,\beta)=(h-1) Char(\mathfrak{g})$, then, both $X(c)_{T}^{ss}(\mathcal{L}_{\alpha_{0}})$ and $X(c^{-1})_{T}^{ss}(\mathcal{L}_{\alpha_{0}})$ are non-empty.

By lemma(3.6), $H^{0}(\tau , \mathfrak{g}/\mathfrak{b})$ is a quotient of 
$\mathfrak{g}$. So, the caharcater of $H^{0}(\tau , \mathfrak{g}/\mathfrak{b})$
must be less than or equal to the character of $\mathfrak{g}$. Further, 
by theorem(3.8), if $\sum_{\beta\in R^{+}(\tau)}h^{0}(\tau,\beta)= Char(\mathfrak{g})$ then, $X(\tau^{-1})_{T}^{ss}(\mathcal{L}_{\alpha_{0}})$ is non-empty. 

Hence, if $\sum_{\tau\in C^{\prime}}\sum_{\beta\in R^{+}(\tau)}h^{0}(\tau,\beta)=(h-1) Char(\mathfrak{g})$, then, we must have 
$\sum_{\beta\in R^{+}(c^{-1})}h^{0}(c^{-1},\beta)=Char(\mathfrak{g})$ and 
$\sum_{\beta\in R^{+}(c)}h^{0}(c,\beta)=Char(\mathfrak{g})$. 

Hence, by using above arguments, we conclude that 
both $X(c)_{T}^{ss}(\mathcal{L}_{\alpha_{0}})$ and 
$X(c^{-1})_{T}^{ss}(\mathcal{L}_{\alpha_{0}})$ are non-empty.

This proves {\it Necessary condition}.
. 
{\it Proof of Sufficient condition:}

We now prove that if both $X(c)_{T}^{ss}(\mathcal{L}_{\alpha_{0}})$ and $X(c^{-1})_{T}^{ss}(\mathcal{L}_{\alpha_{0}})$ are non-empty, then, $\sum_{\tau\in C^{\prime}}\sum_{\beta\in R^{+}(\tau)}h^{0}(\tau,\beta)
=(h-1) Char(\mathfrak{g})$.

Now, if both $X(c)_{T}^{ss}(\mathcal{L}_{\alpha_{0}})$ and $X(c^{-1})_{T}^{ss}(\mathcal{L}_{\alpha_{0}})$ are non-empty, then by [10, theorem (4.2)], we have 
$G$ must be of type $A_{n}$ and either $c=s_{\alpha_{n}}s_{\alpha_{n-1}}\cdots s_{\alpha_{1}}$ or $c^{-1}=s_{\alpha_{n}}s_{\alpha_{n-1}}\cdots s_{\alpha_{1}}$.

With out loss of generality, we may assume that 
$c=s_{\alpha_{n}}s_{\alpha_{n-1}}\cdots s_{\alpha_{1}}$. Now, a simple calculation shows that $c^{r}=w_{\alpha_{r}}$ for every $r=1, 2, \cdots n$.

Hence, we have $C^{\prime}=\{w_{\alpha_{r}}: r=1, 2, \cdots n\}$.

Now, proof of {\it Sufficient condition} follows by using theorem(4.2). 

This completes the proof of corollary.

\end{proof}

Let $c$ be a Coxeter element of $W$. We choose an ordering  
$\{\alpha_{1}, \alpha_{2} , \cdots \alpha_{n} \}$ of simple 
roots such that $c=s_{\alpha_{n}}s_{\alpha_{n-1}}\cdots s_{\alpha_{2}}
s_{\alpha_{1}}$ is a reduced expression for $c$.

By proposition(5.1), for every $j\in \{1, 2, \cdots n\}$, there is a positive 
integer $m_{j}$ such that $c^{m_{j}}(\omega_{\alpha_{j}})=
w_{0}(\omega_{\alpha_{j}})$. For this choice of $m_{j}$, we have 
$$1=\langle \omega_{\alpha_{j}}, \alpha_{j} \rangle=
\langle  c^{m_{j}}(\omega_{\alpha_{j}}), c^{m_{j}}(\alpha_{j}) \rangle 
 =\langle w_{0}(\omega_{\alpha_{j}}), c^{m_{j}}(\alpha_{j}) \rangle.$$ 
 
Hence, we see that $c^{m_{j}}(\alpha_{j})$ is a negative root.

To prove theorem C, we now proceed as follows:

Let $J^{\prime}$ denote the set of all integers  $j$ in 
$\{1, 2, \cdots n \}$ for which there is a positive integer $a_{j}$ 
such that $c^{i}(\alpha_{j})$ is a simple root for every 
$i=0, 1, 2, \cdots a_{j}-1$ and 
$c^{a_{j}}(\alpha_{j})$ is a negative root. 
Let $J$ denote the set of all elements $j$ in $J^{\prime}$
such that $c^{-1}(\alpha_{j})$ is not a simple root.

The following three lemmas describe some properties of the set $J$ 
which will be used in the proof of theorem(5.7).

\begin{lemma}
Let $i$ and $j$ be two distinct elements of $\{1, 2, \cdots n\}$.
Then, $c(\alpha_{i})=\alpha_{j}$ if and only if the following holds:

\begin{enumerate}
\item
$j$ is the unique element in $\{1,2, \cdots i-1\}$ such that 
$\langle \alpha_{j} , \alpha_{i} \rangle \neq 0$. \\  
\item
$i$ is the unique element in $\{j+1,j+2, \cdots n\}$ such that 
$\langle \alpha_{j} , \alpha_{i} \rangle \neq 0$. \\  
\end{enumerate}
\end{lemma}
\begin{proof}
Proof follows by our chosen ordering
$\{\alpha_{1}, \alpha_{2} , \cdots \alpha_{n} \}$ of simple 
roots such that $c=s_{\alpha_{n}}s_{\alpha_{n-1}}\cdots s_{\alpha_{2}}
s_{\alpha_{1}}$ is a reduced expression for $c$.
\end{proof}

We also have

\begin{lemma}
Let $i$ and $j$ be two distinct elements of $J$.
Then, we have 
$\langle  c^{p}(\alpha_{j}), c^{q}(\alpha_{k}) \rangle =0$ for any two 
distinct elements $j$ and $k$ of $J$ and for every 
$p=0, 1, 2, \cdots a_{j}-1$ and for every  $q=0, 1, 2,  \cdots a_{k}-1$.
\end{lemma}
\begin{proof}

By the definition of $J$, we can see that 

{\it Observation 1:}
 
$c^{m}(\alpha_{j})\neq \alpha_{k}$ 
for any $m=0, 1, \cdots a_{j}-1$ and 
$c^{m}(\alpha_{k})\neq \alpha_{j}$ 
for any $m=0, 1, \cdots a_{k}-1$. 

For instance, we can prove {\it Observation } as follows:

Since $k\in J$, we see that $c^{t}(\alpha_{k})$
is a simple root for every $t=0, 1, \leq m$. On the other hand,
since $j\in J$, $c^{-1}(\alpha_{j})$ is not simple. Thus, we have 
$p=q$. This is a conradiction as $j\neq k$.

We now proceed to prove the lemma.

With out loss of generality, we may asuume that $p\leq q$.

Hence, we have 
$$\langle  c^{p}(\alpha_{j}), c^{q}(\alpha_{k})\rangle =\langle \alpha_{j}, 
c^{q-p}(\alpha_{k})\rangle.$$ 

If $\langle \alpha_{j} , c^{q-p}(\alpha_{k})\rangle$ is non-zero, either  
we have $\alpha_{j}=c^{q-p}(\alpha_{k})$ or 
$\langle \alpha_{j}, c^{q-p}(\alpha_{k}) \rangle=-1$.

$\alpha_{j}=c^{q-p}(\alpha_{k})$ is not possible by {\it Observation 1}.

So, we may assume that 
$\langle \alpha_{j}, c^{q-p}(\alpha_{k}) \rangle=-1$.
Let $c^{q-p}(\alpha_{k})=\alpha_{t}$ 
for some $t\in \{1, 2, \cdots n\}$.

With out loss of generality, we may assume that $j < t$.
Since $\langle \alpha_{j} , \alpha_{t}\rangle = \langle \alpha_{j} , c^{q-p}(\alpha_{k})\rangle$ is non-zero,  $c(\alpha_{t})$ must be positive. 
Since $q-p\leq a_{k}-1$, $c(\alpha_{t})$ must be a simple root. Since 
$\langle \alpha_{j} , \alpha_{t}\rangle $ is non-zero, we have 
$c(\alpha_{t})=\alpha_{j}$. Hence, we have 
$c^{q+1-p}(\alpha_{k})=\alpha_{j}$. This is not possible by {\it Observation 1}.

This completes proof of the lemma.

\end{proof}

For any $j\in J$,  we take $\phi_{j}=s_{\alpha_{j}} s_{c(\alpha_{j})}\cdots 
s_{c^{a_{j}-1}(\alpha_{j})}$. 

Then, we have 

\begin{lemma}
\begin{enumerate}
\item
$\phi_{j}$ commutes with  $\phi_{k}$ for any $j$ and $k$ in $J$.\\  
\item
Let $\phi$ denote the product $\Pi_{j\in J}\phi_{j}$. Then, 
we can write $c=\tau\phi$ such that 
$l(c)=l(\phi)+l(\tau)$. \\
\item
Let $r\in \{1, 2, \cdots n\}$ be such that $s_{\alpha_{r}}\leq \tau$.
Then, we have 
$height(c(\alpha_{r}))\geq height(\phi(\alpha_{r}))$.\\
\end{enumerate}
\end{lemma}
\begin{proof}

{\it Proof of (1):}

Proof of (1) follows from lemma(5.5).

{\it Proof of (2):}

Proof of (2) follows from the fact that $R^{+}(\phi)\subset R^{+}(c)$.

{\it Proof of (3):}

Since $s_{\alpha_{r}}\leq \tau$, we have 
$s_{\alpha_{r}}\leq c$ and $s_{\alpha_{r}}\nleq \phi$.
Hence, $\phi(\alpha_{r})$ is a positive root. 

Since $s_{\alpha_{r}}\nleq \phi$, $c(\alpha_{r})\neq c^{t}(\alpha_{j})$ for any 
$j\in J$ and for any $t=0, 1, \cdots a_{j}-1$.

If $\phi(\alpha_{r})=\alpha_{r}$, then, we have  

$$height(c(\alpha_{r}))=1=height(\alpha_{r})=height(\phi(\alpha_{r})).$$

Further, if $c(\alpha_{r})$ is a simple root, then, 
$\phi(\alpha_{r})=\alpha_{r}$. Proof is similar to the above case.

So, we may assume that  $c(\alpha_{r})$ is not a simple root and 
$\phi(\alpha_{r})\neq \alpha_{r}$. Hence, we can use lemma(5.4) to conclude
that either there are two distict 
elements $j$ and $k$ in $\{1,2, \cdots r-1\}$ such that both 
$\langle \alpha_{j} , \alpha_{r} \rangle$ and 
$\langle \alpha_{k} , \alpha_{r} \rangle$ are non-zero 
or there is a positive integer $j < r$ such that $s_{\alpha_{j}}\leq\phi$ and a 
positive integer $k>j$, $k\neq r$ such that both 
$\langle \alpha_{j} , \alpha_{r} \rangle$ and 
$\langle \alpha_{k} , \alpha_{j} \rangle$ are non-zero.

If there are two distict elements $j$ and $k$ in $\{1,2, \cdots r-1\}$ 
such that both $\langle \alpha_{j} , \alpha_{r} \rangle$ and 
$\langle \alpha_{k} , \alpha_{r} \rangle$ are non-zero, then,  
 we must have $c(\alpha_{r})\geq \phi(\alpha_{r})$.
Hence, we have $height(c(\alpha_{r}))\geq height(\phi(\alpha_{r}))$.

Otherwise, we use lemma(5.4) to conclude that 
$c(\alpha_{k})\neq \alpha_{j}$. Hence, we have
$s_{\alpha_{k}}\nleq\phi$. 
Thus, we have $c(\alpha_{r})+\alpha_{r}\geq \phi(\alpha_{r})+\alpha_{k}$.

Hence, we have $height(c(\alpha_{r}))\geq height(\phi(\alpha_{r}))$.

This completes the proof of (3).

\end{proof}

\subsection{Proof of theorem C}

In this subsection , we prove theorem C as follows:

\begin{theorem}
\begin{enumerate}
\item
Let $\tau\in W$. The cohomology module $H^{l(\tau)}(\tau, \tau^{-1}\cdot 0)$ 
is the one dimensional trivial representation of $B$. 
\item
Let $c$ be a Coxeter element of $W$. Then, 
$H^{i}(c, c^{-1}\cdot 0)$ is zero for every $i\neq l(c)$ if and only if both 
$X(c)_{T}^{ss}(\mathcal{L}_{\alpha_{0}})$ and 
$X(c^{-1})_{T}^{ss}(\mathcal{L}_{\alpha_{0}})$ are non-empty.
\end{enumerate}
\end{theorem}
\begin{proof}

{\it Proof of (1):} 

Since $\tau\cdot\tau^{-1}\cdot 0=0$, by the 
Borel-Weil-Bott's theorem, $H^{l(\tau)}(w_{0}, \tau^{-1}\cdot 0)$ is the one
dimensional trivial representation of $G$. On the otherhand, 
by [9, Proposition (4.2)], 
$H^{l(\tau)}(\tau, \tau^{-1}\cdot 0)$ is non-zero. 

By [9, corollary (4.3)], the restriction map 
$H^{l(\tau)}(w_{0}, \tau^{-1}\cdot 0)\longrightarrow 
H^{l(\tau)}(\tau, \tau^{-1}\cdot 0)$ is surjective. Thus, 
$H^{l(\tau)}(\tau, \tau^{-1}\cdot 0)$ is the one dimensional trivial 
representation of $G$. 
This proves (1).

{\it Proof of (2):}

{\it Proof of sufficinet condition:}

We first prove that if both 
$X(c)_{T}^{ss}(\mathcal{L}_{\alpha_{0}})$ and 
$X(c^{-1})_{T}^{ss}(\mathcal{L}_{\alpha_{0}})$ are non-empty, then, 
$H^{i}(c, c^{-1}\cdot 0)$ is zero for every $i\neq l(c)$.

We now assume that both $X(c)_{T}^{ss}(\mathcal{L}_{\alpha_{0}})$ and 
$X(c^{-1})_{T}^{ss}(\mathcal{L}_{\alpha_{0}})$ are non-empty.

Then by [10, theorem (4.2)], we have 
$G$ must be of type $A_{n}$ and either $c=s_{\alpha_{n}}s_{\alpha_{n-1}}\cdots s_{\alpha_{1}}$ or $c^{-1}=s_{\alpha_{n}}s_{\alpha_{n-1}}\cdots s_{\alpha_{1}}$.

With out loss of generality, we may assume that 
$c=s_{\alpha_{n}}s_{\alpha_{n-1}}\cdots s_{\alpha_{1}}$. 

{\it Step 1} We show that  $c^{r}=w_{\alpha_{r}}$ for every $r=1, 2, \cdots n$.

Using a simple computation, we see that 
$c(\alpha_{1})=-(\sum_{t=1}^{n}\alpha_{t})$ and that $c(\alpha_{j})=\alpha_{j-1}$
for every $j=2, 3, \cdots n$.

Now, let $\overline{m}$ denote the remainder when $m$ is divided by 
$n+1$.

Using recursion on $r$, we can show that 
$c^{r}(\alpha_{r})=-(\sum_{t=1}^{n}\alpha_{t})$ and
$c^{r}(\alpha_{j})=\alpha_{\overline{n+1+j-r}}$ for every $j\neq r$.

Hence, we have  $R^{+}(c^{r})=\{\beta\in R^{+}: \beta \geq \alpha_{r} \}$
for every $r=1, 2, \cdots n$. On the otherhand, we have 
$R^{+}(w_{\alpha_{r}})=\{\beta\in R^{+}: \beta \geq \alpha_{r} \}$.  
Thus, we have $c^{r}=w_{\alpha_{r}}$ for every $r=1, 2, \cdots n$.

This proves {\it Step 1}.

We consider the natural projection $\pi_{r}:G/B\longrightarrow G/Q_{\alpha_{r}}$
given by $\pi_{r}(xB)=xQ_{\alpha_{r}}$.

Since $R^{+}(w_{\alpha_{r}})=\{\beta\in R^{+}: \beta \geq \alpha_{r} \}$,  
$w_{\alpha_{r}}^{-1}\cdot 0$ is equal to $\sum_{\beta\geq \alpha_{r}}-\beta$
of all negative roots $-\beta$ such that $\beta\geq \alpha_{r}$.

On the otherhand, we have $\sum_{\beta\geq \alpha_{r}}-\beta$
is equal to the multiple $(n+1)\omega_{\alpha_{r}}$ of the fundamental weight
$\omega_{\alpha_{r}}$ corresponding to the simple root $\alpha_{r}$.

Since $\mathcal{L}_{-(n+1)\omega_{\alpha_{r}}}$ is the canonical line bundle 
on $G/Q_{\alpha_{r}}$, $H^{i}(G/Q_{\alpha_{r}}, \mathcal{L}_{-(n+1)\omega_{\alpha_{r}}})$
vanishes for every $i\neq dim(G/Q_{\alpha_{r}})$. Since 
$dim(G/Q_{\alpha_{r}})=l(w_{\alpha_{r}})$, 
$H^{i}(G/Q_{\alpha_{r}}, \mathcal{L}_{-(n+1)\omega_{\alpha_{r}}})$
vanishes for every $i\neq l(w_{\alpha_{r}})$. 

Thus, we have 

{\it Observation 1:} $H^{i}(G/Q_{\alpha_{r}}, \mathcal{L}_{w_{\alpha_{r}}^{-1}\cdot 0})$
vanishes for every $i\neq l(w_{\alpha_{r}})$. 

Since the restriction of $\pi_{r}$ to $X(w_{\alpha_{r}})$ is a 
birational morphism, the pull back map 
$\pi_{r}^{*}:H^{i}(w_{\alpha_{r}}, w_{\alpha_{r}}^{-1}\cdot 0)\longrightarrow 
H^{i}(G/Q_{\alpha_{r}}, \mathcal{L}_{w_{\alpha_{r}}^{-1}\cdot 0})$
is an isomorphism for every $i$. 

{\it Proof of sufficinet condition} follows from {\it Observation 1}. 

{\it Proof of necessary condition:}

We now prove the necessary condition.

Let $c$ be a Coxeter element of $W$ such that $H^{i}(c, c^{-1}\cdot 0)$ is zero
for every $i\neq l(c)$. Firstly, we can find an ordering 
$\{\alpha_{1}, \alpha_{2} , \cdots \alpha_{n} \}$ of simple 
roots such that $c=s_{\alpha_{n}}s_{\alpha_{n-1}}\cdots s_{\alpha_{2}}
s_{\alpha_{1}}$ is a reduced expression for $c$.

Let $J$ be as in subsection 5.1. For any $j\in J$, we take 
$\phi_{j}=s_{\alpha_{j}} s_{c(\alpha_{j})}\cdots s_{c^{a_{j}-1}(\alpha_{j})}$ as in 
subsection 5.1. By lemma(5.6), we see that $\phi_{j}$ commutes with $\phi_{k}$ 
for any $j$ and $k$ in $J$.  

As in lemma(5.6), we denote the product $\Pi_{j\in J}\phi_{j}$ of the 
$\phi_{j}$'s by $\phi$. as in lemma(5.6), we can write 
$c$ as a product $c=\tau\phi$ with $l(c)=l(\phi)+l(\tau)$.

{\it Claim:}

We first show that $c^{-1}\cdot 0  - \phi_{j}^{-1}\cdot 0$- weight space of 
$H^{l(\phi_{j})}(\phi_{j}, c^{-1}\cdot 0)$ is non-zero.

We first note that $-\phi_{j}^{-1}\cdot 0$ is equal to the sum of all
positive roots $\beta$ in $R^{+}(\phi_{j})$. The set $R^{+}(\phi_{j})$ 
consists of precisely the roots of the  form 
$\sum_{i=r}^{a_{j}-1}c^{i}(\alpha_{j})$, $r$-running over all integers from $0 , 
1, \cdots , a_{1}-1$.   

Since $c^{i}(\alpha_{j})$ is a simple root for every 
$i=0, 1, 2, \cdots a_{j}-1$ and $c^{a_{j}}(\alpha_{j})$ is a negative root,
we have $c^{a_{j}}(\alpha_{j}) \leq - \sum_{i=0}^{a_{j}-1}c^{i}(\alpha_{j})$.
On the otherhand, we have $c^{-1}\cdot 0= c^{-1}(\rho) - \rho$.

Hence, we have 
$$\langle  c^{-1}\cdot 0 , c^{a_{j}-1}(\alpha_{j}) \rangle = 
\langle \rho  , c(c^{a_{j}-1}(\alpha_{j}))  \rangle  - \langle \rho  , 
c^{a_{j}-1}(\alpha_{j}) \rangle  \leq \langle \rho  , 
\sum_{i=0}^{a_{j}-1}c^{i}(\alpha_{j}) \rangle - 1   =  -a_{j}-1.$$

For simplicity of notation, we let $\gamma_{i}= c^{i}(\alpha_{j})$
for every $i=0,1, 2, \cdots a_{j}-1$.

By lemm(2.3), $c^{-1}\cdot 0 + a_{j} \gamma_{a_{j}-1}$-weight 
space of $H^{1}(s_{\gamma_{a_{j}-1}}, c^{-1}\cdot 0)$ is non-zero.

Further, we have,  
$$\langle  c^{-1}\cdot 0 + a_{j} \gamma_{a_{j}-1} , \gamma_{a_{j}-2} 
\rangle = \langle \rho  , c(\gamma_{a_{j}-2})  \rangle  - \langle \rho  , 
\gamma_{a_{j}-2} \rangle -a_{j}= -a_{j}.$$

Since each $\gamma_{a_{j}-2}$-string of weights of 
$H^{1}(s_{\gamma_{a_{j}-1}}, c^{-1}\cdot 0)$ is of length one, each indecomposable 
$B_{\gamma_{a_{j}-2}}$- module is one dimensional.

Inparticular, the one dimensional $\mathbb{C}_{c^{-1}\cdot0 +a_{j}\gamma_{a_{j}-1}}$
is $B_{\gamma_{a_{j}-2}}$-direct summand of 
$H^{1}(s_{\gamma_{a_{j}-1}}, c^{-1}\cdot 0)$.

Using the same argument again, we see that 
the one dimensional space $\mathbb{C}_{c^{-1}\cdot 0 +a_{j}\gamma_{a_{j}-1}+(a_{j}-1)\gamma_{a_{j}-2}}$ is $B_{\gamma_{a_{j}-3}}$-direct summand of 
$H^{2}(s_{\gamma_{a_{j}-2}}s_{\gamma_{a_{j}-1}} , c^{-1}\cdot 0 )$.

Proceeding recursively, we can show that 
$c^{-1}\cdot 0  - \sum_{i=0}^{a_{j}-1}(i+1)\gamma_{i}$- weight space of 
$H^{l(\phi_{j})}(\phi_{j}, c^{-1}\cdot 0)$ is non-zero.

Since $ \sum_{i=0}^{a_{j}-1}(i+1)\gamma_{i} = \phi_{j}^{-1}\cdot 0$,
it follows that $c^{-1}\cdot 0  - \phi_{j}^{-1}\cdot 0$- weight space of 
$H^{l(\phi_{j})}(\phi_{j}, c^{-1}\cdot 0)$ is non-zero.

Using the same process, we can show that 
$c^{-1}\cdot 0  - \phi^{-1}\cdot 0$- weight space of 
$H^{l(\phi)}(\phi, c^{-1}\cdot 0)$ is non-zero.

This proves the {\it Claim}.

We now prove that the $c^{-1}\cdot 0 - \phi^{-1}\cdot 0$- weight space of 
$H^{l(\phi)}(c , c^{-1}\cdot 0)$ is non-zero. 

Let $r\in \{1, 2, \cdots n\}$ be such that $s_{\alpha_{r}}\leq \tau$.
That is, $s_{\alpha_{r}}\leq c$ but $s_{\alpha_{r}}\nleq \phi$.

$$\langle  c^{-1}\cdot 0 - \phi^{-1}\cdot 0 , \gamma_{r} \rangle = \langle \rho , c(\alpha_{r}) \rangle - \langle \rho  , \phi(\alpha_{r}) \rangle 
= height(c(\alpha_{r})) -height(\phi(\alpha_{r})).$$ 

Hence, using lemma(5.6), we see that 
$height(c(\alpha_{r}))\geq height(\phi(\alpha_{r}))$.

Thus, we conclude that the line bundle 
$\mathcal{L}_{\mathbb{C}_{c^{-1}\cdot 0  - \phi^{-1}\cdot 0}}$ corresponding to the 
one dimensional $B_{\alpha_{r}}$-module 
$\mathbb{C}_{c^{-1}\cdot 0  - \phi^{-1}\cdot 0}$ is an effective line bundle 
on $P_{\alpha_{r}}/B$. 

Hence, $\mathbb{C}_{c^{-1}\cdot 0  - \phi^{-1}\cdot 0}$ is a 
$B_{\alpha_{t}}$-direct summand of $H^{0}(s_{\alpha_{r}}, H^{l(\phi)}(\phi, c^{-1}\cdot 0)$ for every $t\neq r$ such that $s_{\alpha_{t}}\leq \tau$.

Using the same argument recursively, we conclude that the 
$c^{-1}\cdot 0  - \phi^{-1}\cdot 0$-weight space of 
$H^{0}(\tau , H^{l(\phi)}(\phi, c^{-1}\cdot 0))$ is non-zero.

Using {\it SES}, we see that 
$c^{-1}\cdot 0  - \phi^{-1}\cdot 0$-weight space of 
$H^{l(\phi)}(c, c^{-1}\cdot 0))$ is non-zero.

On the otherhand, by hypothesis, we have $H^{i}(c, c^{-1}\cdot 0)$ is zero 
for every $i\neq l(c)$. Hence, we have $\phi=c$. Since $G$ is simple, the 
Dynkin diagram of the Lie algebra $\mathfrak{g}$ of $G$ is connected.
Hence, $J$ can have only point $\{j\}$ as the $\phi_{i}$'s commute with 
each other.

Further, by our chosen reduced expression $c=s_{\alpha_{n}}s_{\alpha_{n-1}}\cdots s_{\alpha_{2}}s_{\alpha_{1}}$ for $c$, $j$ must be $n$ and 
$c^{k}(\alpha_{n})=\alpha_{n-k}$ for every $k=0,1, \cdots n-1$.

Hence $G$ must be of type $A_{n}$ and the ordering of the simple roots is 
simply the ordering in its Dynkin diagram.

This proves the {\it necessary condition}.

This completes the proof of theorem.

\end{proof}

Let $c$ be a Coxeter element of $W$. Let $C$ denote the cyclic subgroup of 
$W$ generated by $c$. Then, the order of $C$ is the Coxeter number and we 
denote it by $h$. The sum $\sum_{\tau\in C}\sum_{j=0}^{l(\tau)}h^{j}(\tau,\tau^{-1}\cdot 0)$ of the chracters $h^{j}(\tau,\tau^{-1}\cdot 0)$ of the cohomolgy 
modules $H^{j}(\tau,\tau^{-1}\cdot 0)$, $\tau$ running over all elements of 
$C$ and $j$ running over all integers from $\{0, 1, 2, \cdots l(\tau)\}$
is an element of the representation ring $\mathbb{Z}[X(T)]$ of $T$.

The following corollary is another application of our main theorem. 
 
\begin{corollary}
$\sum_{\tau\in C}\sum_{i=0}^{l(\tau)}h^{i}(\tau,\tau^{-1}\cdot 0)$ is equal
to the Coxeter number $h$ if and only if both 
$X(c)_{T}^{ss}(\mathcal{L}_{\alpha_{0}})$ and 
$X(c^{-1})_{T}^{ss}(\mathcal{L}_{\alpha_{0}})$ are non-empty.
\end{corollary}
\begin{proof}

By theorem(5.5(1)), $H^{l(\tau)}(\tau, \tau^{-1}\cdot 0)$ 
is the one dimensional trivial representation of $B$
for every element $\tau$ in $W$. Therefore, 
the sum $\sum_{\tau\in C}\sum_{i=0}^{l(\tau)}h^{l(\tau)}(\tau,\tau^{-1}\cdot 0)$ 
is equal to $h.1$.

Again using theorem (5.5(2)), we see that $\sum_{\tau\in C}\sum_{i=0}^{l(\tau)-1}h^{i}(\tau,\tau^{-1}\cdot 0)$ is zero if and only both 
$X(c)_{T}^{ss}(\mathcal{L}_{\alpha_{0}})$ and 
$X(c^{-1})_{T}^{ss}(\mathcal{L}_{\alpha_{0}})$ are non-empty.

This completes proof of corollary.

\end{proof}

\section{Proof of theorem B }

Throughout this section, we assume that $G$ is simple, simply connected
algebraic group over $\mathbb{C}$ which is not simply laced.

We first prove that when $G$ is not simply laced, then , there is a 
positive root $\beta$ and a simple root $\alpha$ such that $s_{\alpha} \cdot \beta$ is the highest short root. 

\begin{lemma}
Let $G$ be a simple algebraic group which is not simply laced.
Then, there is a positive root $\beta$ and a simple root $\alpha$ such 
that $s_{\alpha}\cdot \beta$ is the highest short root. 
\end{lemma}
\begin{proof}

If $G$ is of type $G_{2}$, then, the simple roots $\alpha_{1}$ and 
$\alpha_{2}$ satisfy the following:

$\langle \alpha_{1} , \alpha_{2} \rangle =-1$ and 
$\langle \alpha_{2} , \alpha_{1} \rangle =-3$.

Here, we follow the convention in [6].

In this case, we take $\beta=\alpha_{2}$ and $\alpha=\alpha_{1}$.
Hence, $s_{\alpha}\cdot \beta=\alpha_{2}+2\alpha_{1}$ is the highest 
short root. 
 
Hence, we may assume that $G$ is a simple algebraic group of type 
$B_{n}$, $C_{n}$ or $F_{4}$.

Let $\nu$ be a the highest short root. 
We now show that there is a simple root $\alpha$ such that 
$\nu+\alpha$ is a root and $\langle \nu , \alpha \rangle  = 0$.

To show that $\nu+\alpha$ is a root, it is sufficient to show that 
the weight space $\mathfrak{g}_{\nu+\alpha}$ is non-zero.

On the otherhand, $\mathfrak{g}_{\nu+\alpha}$ is non-zero is a statement 
independent of the characteristic. So, we may assume that 
$k=\mathbb{C}$. Hence, $\mathfrak{g}$ is an irreducible $G$-module.

Hence, $\mathfrak{g}_{\alpha_{0}}$ is the unique $B^{+}$-stable line in
$\mathfrak{g}$. Hence, there is a simple root $\alpha$ such that 
$ad\mathfrak{X}_{\alpha}(\mathfrak{g}_{\nu})$ is non-zero.
Thus, $\nu+\alpha$ is a root. 

Since $\nu$ is dominant, we have 

{\it Observation 1}

$\langle \nu , \alpha \rangle \geq 0$.

On the other hand, since $G$ is not of type $G_{2}$, 
$\langle \nu +\alpha , \alpha \rangle \leq 2$.
Hence, we have $\langle \nu , \alpha \rangle \leq  0$
By {\it Observation }, we have 
$\langle \nu , \alpha \rangle  = 0$.

Proof of the lemma follows by taking $\beta=s_{\alpha}\cdot\nu$.

\end{proof}

Let $\alpha$ and $\beta$ be as in lemma(6.1). 

Let $\tau\in W$ be such that $s_{\alpha}\leq \tau$. 

Let $V=\bigoplus_{\mu\leq \beta}\mathfrak{g}_{\mu}$ be the direct sum of all 
$T$-weight spaces  of weights $\mu$ satisfying $\mu\leq \beta$. Clearly $V$ is a
 $B$-sub module of $\mathfrak{g}$ 
containing $\mathfrak{b}$. 

Then, we have 

\begin{lemma}
The $s_{\alpha}\cdot \beta$- weight space of   
$H^{1}(\tau , V)$  is a non-zero.  
\end{lemma}
\begin{proof}

Since $s_{\alpha}\cdot \beta$ is the highest short root, 
$s_{\alpha}\cdot \beta$ is dominant character of $T$.  
Hence, by the Borel-Weil-Bott's theorem, 
$H^{1}(w_{0}, s_{\alpha}\cdot \beta)$ is an irreducible representation of $G$
with highest weight $s_{\alpha}\cdot \beta$.

On the otherhand, by [9, Proposition (4.2)], 
$H^{1}(s_{\alpha}, s_{\alpha}\cdot \beta)$ is non-zero. 

By [9, corollary (4.3)], the restriction map 

$H^{1}(w_{0}, s_{\alpha}\cdot \beta) \longrightarrow 
H^{1}(s_{\alpha}, s_{\alpha}\cdot \beta)$ is surjective. 

Thus, the restriction map $H^{1}(\tau,  s_{\alpha}\cdot \beta)
\longrightarrow H^{1}(s_{\alpha}, s_{\alpha}\cdot \beta)$ is also 
surjective.

This completes proof the lemma.

\end{proof}

Let $G$ be a simple, simply connected algebraic group over $\mathbb{C}$
which is not simply laced.

Let $V_{1}$ be a $B$-sub module of $\mathfrak{g}$ containing $\mathfrak{b}$.
Let $V_{2}$ be a $B$-sub module of $V_{1}$ containing $\mathfrak{b}$.
Let $\tau\in W$.

The natural projection $\Pi:V_{1} \longrightarrow 
V_{1}/V_{2}$ of $B$-modules induces a 
homomorphism of $B$-modules $\Pi_{\tau}:H^{0}(\tau, V_{1})
\longrightarrow H^{0}(\tau, V_{1}/V_{2})$ of $B$-modules.

We now deduce the following lemma as a consequence of the above lemma.

Let $\tau\in W$. Let $\gamma$ be a simple root. 
Let $V$ be a $B$-sub module of $\mathfrak{g}$ 
containing $\mathfrak{b}$. {\it Then, we have the following lemma on 
indecomposable sub modules of $H^{1}(\tau, V)$ similar to lemma(3.4) except  
that (2) and (5) are possible only if $\gamma$ is a short root }

\begin{lemma}
Every indecomposable $B_{\gamma}$- summand $V^{\prime}$ of 
$H^{1}(\tau , V)$ must be one of the following:
\begin{enumerate}
\item
$V^{\prime}=\mathbb{C}\cdot H$ for some $H\in \mathfrak{h}$ such that 
$\gamma(H)=0$. \\
\item  
$V^{\prime}=\mathbb{C}\cdot H \bigoplus \mathfrak{g}_{-\gamma}$ for some $H\in \mathfrak{h}$ such that $\gamma(H)=1$ and 
$\nu(H)=0$ for every simple root $\nu$ different from $\gamma$. \\
\item
$V^{\prime}=\mathfrak{g}_{\beta}$ for some short root $\beta$ 
different from $\gamma$. \\
\item  $V^{\prime}=\mathfrak{g}_{\beta} \bigoplus \mathfrak{g}_{\beta-\gamma}$ 
for some short root $\beta$. \\ 
\item
$V^{\prime}=sl_{2,\gamma}$, the three dimensional irreducible 
$L_{\gamma}$-module with highest weight $\gamma$.\\
\end{enumerate}
\end{lemma}
\begin{proof}

Proof is by induction on $l(\tau)$.

If $l(\tau)=0$, then, $\tau=id$ and so we are done.

Otherwise, choose a simple root $\alpha$ such that 
$l(\tau)=1+l(s_{\alpha}\tau)$.

By induction, every indecomposable $B_{\gamma}$- summand $V^{\prime}$ of 
$H^{1}(s_{\alpha}\tau, V)$ must be one of the above mentioned 5 types.

We first assume that $\gamma=\alpha$. 

Now, using lemma(2.3(1)), we see that if $V^{\prime}$ is one of 
the types (1), (4), and (5), then, $H^{0}(s_{\gamma}, V^{\prime})$ must be 
of the same type in $H^{0}(s_{\gamma}, H^{1}(s_{\gamma}\tau, V))$.
In $V^{\prime}$ is of type (2), using lemma(2.3(3)), we see that 
$H^{0}(s_{\gamma}, V^{\prime})=(0)$.

In type (3), using lemma(2.3), we see that $H^{0}(s_{\gamma}, V^{\prime})$ is 
either zero or is one of the types (3) or (4).

We may therefore assume that $\alpha\neq\gamma$.

In this case, we use lemma(2.3) to see that if $V^{\prime}$ is of type 
different from type (5), then  $H^{0}(s_{\alpha}, V^{\prime})$ must be of the 
same type in $H^{0}(s_{\alpha}, H^{1}(s_{\alpha}\tau, V))$.

If $V^{\prime}$ is of type (5), we again use lemma (2.3) to conclude 
that $H^{0}(s_{\alpha}, V^{\prime})$ must be of type (2) in
$H^{0}(s_{\alpha}, H^{1}(s_{\alpha}\tau, V))$.

(Here, we use the induction hypothesis that $V^{\prime}$ can be of type(5) 
only if $\gamma$ is a short root. Hence, we have
$\langle \gamma , \alpha \rangle =-1$)

We recall {\it SES}: 

$$
(0) \longrightarrow H^{1}(s_{\gamma}, H^{0}(s_{\gamma}\tau, V) ) 
\longrightarrow H^{1}(\tau , V) \longrightarrow
H^{0}(s_{\gamma} , H^{1}(s_{\gamma}\tau, V )) \longrightarrow (0).$$

This completes the proof for the case when $\alpha=\gamma$.

We may therefore assume that $\alpha\neq\gamma$.

In this case, we use lemma(2.3) to see that if $V^{\prime}$ is of type 
different from type (5), then  $H^{0}(s_{\alpha}, V^{\prime})$ must be of the 
same type in $H^{0}(s_{\alpha}, H^{1}(s_{\alpha}\tau, V))$.

If $V^{\prime}$ is of type (5), we again use lemma (2.3) to conclude 
that $H^{0}(s_{\alpha}, V^{\prime})$ must be of type (2) in
$H^{0}(s_{\alpha}, H^{1}(s_{\alpha}\tau, V))$.

(This is because $\langle \gamma , \alpha \rangle =-1$)

Proof the lemma is completed using {\it SES}.

\end{proof}

For any $B$ module $V$, and for any $\tau\in W$, 
we recall the evaluation map $ev:H^{0}(\tau,V)\longrightarrow V$
by $ev(s)=s(idB)$, the evaluation of the section at the identity
coset $idB$ as point in $X(\tau)$.  

Then, we have

\begin{lemma}
Let $V$ be a $B$-sub module of $\mathfrak{g}$ containing $\mathfrak{b}$.
Then, we have 
\begin{enumerate}
\item
The evaluation map 
$ev:H^{0}(\tau, V)\longrightarrow V$ is injective.\\
\item
$H^{i}(\tau, V)$ is zero for all $i\geq 2$.\\
\end{enumerate}
\end{lemma}
\begin{proof}

Proof of (1) is similar to that of lemma(3.3).

{\it Proof of (2):}

Proof is by induction on $l(\tau)$.

If $l(\tau)=0$, we are done.

Otherwise, choose a simple root $\gamma\in S$ be such that 
$l(s_{\gamma}\tau)=l(\tau)-1$.

By lemma(6.3), every indecomposable $B_{\gamma}$- summand $V^{\prime}$ 
of $H^{1}(s_{\gamma}\tau , V)$ must be one of the 5 types given 
in lemma(6.3).    

Hence, using lemma(2.3), we conclude that 
$H^{i}(P_{\gamma}/B, V^{\prime})$ is zero  
for every indecomposable $B_{\gamma}$-summand $V^{\prime}$ of  
$H^{1}(s_{\gamma}\tau , V))$ and for every $i\geq 1$.

Hence, we have 
$H^{i}(P_{\gamma}/B, H^{1}(s_{\gamma}\tau , V))=(0)$ 
and for every $i\geq 1$.

Since $dim(P_{\gamma}/B)=1$, 
$H^{i}(P_{\gamma}/B , H^{0}(s_{\gamma}\tau , V))=(0)$ for every 
$i\geq 2$.

Thus, we have shown that 

{\it Observation 1}
\begin{enumerate}
\item
$H^{i}(P_{\gamma}/B, H^{1}(s_{\gamma}\tau , V))=(0)$ 
for all $i\geq 1$.
\item
$H^{i}(P_{\gamma}/B , H^{0}(s_{\gamma}\tau , V))=(0)$ for every 
$i\geq 2$.
\end{enumerate}

By induction on $l(\tau)$, we have 

{\it Observation 2}

$H^{i}(s_{\gamma}\tau , V)$ is zero 
for all $i\geq 2$.

Now, using {\it Observation 1}, we conclude that 
$H^{i}(\tau, V)$ is zero for all $i\geq 2$.

We recall {\it SES}: 
For every $i\geq 1$, we have:

$$
(0) \longrightarrow H^{1}(s_{\gamma}, H^{i-1}(s_{\gamma}\tau, V)) 
\longrightarrow H^{i}(\tau , V) \longrightarrow
H^{0}(s_{\gamma} , H^{i}(s_{\gamma}\tau, V )) \longrightarrow (0).$$

Using {\it Observation 1} and {\it Observation 2} in the above short 
exact sequence, we conclude that $H^{i}(\tau , V)=(0)$ for every 
$i\geq 2$. 

This completes the proof of (2).

\end{proof}

We have the following corollary as an application of the lemma.

\begin{corollary}
Let $\tau\in W$. Let $\alpha$ be a positive root.
Then, $H^{i}(\tau, \alpha)=(0)$ for every $i\geq 2$.
\end{corollary}
\begin{proof}
 Let $V_{1}:=\bigoplus_{\mu\leq \alpha}\mathfrak{g}_{\mu}$ denote the 
direct sum of the weight spaces of $\mathfrak{g}$ of weights $\mu$ satisfying
$\mu\leq \alpha$. 

Let $V_{2}:=\bigoplus_{\mu < \alpha}\mathfrak{g}_{\mu}$ denote the 
direct sum of the weight spaces of $\mathfrak{g}$ of weights $\mu$ satisfying
$\mu < \alpha$. 

It is clear that $V_{2}$ is a $B$-sub module of $\mathfrak{g}$ containing $\mathfrak{b}$ and $V_{1}$ is a $B$-sub module of $\mathfrak{g}$ containing $V_{2}$.

Since $\mathfrak{g}_{\alpha}$ is one dimensional and is isomorphic to
the quotient $V_{1}/V_{2}$, we have 

{\it Observation 1}

$H^{i}(\tau, \alpha)= H^{i}(\tau, \mathfrak{g}_{\alpha})=
H^{i}(\tau, V_{1}/V_{2})$ for every $i\geq 2$.

We have the short exact sequence    
 $(0)\longrightarrow V_{2} \longrightarrow V_{1} \longrightarrow 
V_{1}/V_{2} \longrightarrow (0)$ of $B$- modules.

Applying $H^{i}(\tau, -)$ to this short
exact sequence of $B$-modules,  we obtain 
the following long exact sequence of $B$-modules: 

{\it Observation 2}

$$\cdots H^{i}(\tau, V_{2}) \longrightarrow H^{i}(\tau, V_{1}) 
\longrightarrow H^{i}(\tau, V_{1}/V_{2})\longrightarrow 
H^{i+1}(\tau, V_{2}) \cdots$$

By lemma(6.4), $H^{i}(\tau, V_{1})$ and  
$H^{i+1}(\tau, V_{2})$ are all zero for every $i\geq 2$.
Using {\it Observation 2}, we conclude that $H^{i}(\tau, V_{1}/V_{2})=(0)$
for every $i\geq 2$. 

Proof of corollary follows by applying {\it Observation 1}.

\end{proof}

We now prove theorem B.

\begin{theorem}

Let $G$ be a simple, simply connected but not simply laced algebraic group 
over $\mathbb{C}$. 
Let $\tau \in W$. Then, we have
\begin{enumerate}
\item
$H^{i}(X(\tau), \mathcal{T}_{G/B})=(0)$ for every 
$i\geq 1$. 
\item
The adjoint representation $\mathfrak{g}$ is a $B$-submodule of 
$H^{0}(X(\tau) , \mathcal{T}_{G/B})$  if and 
only if the set of semi-stable points 
$X(\tau^{-1})_{T}^{ss}(\mathcal{L}_{\alpha_{0}})$ is non-empty.
\end{enumerate}
\end{theorem}
\begin{proof}
Proof is similar to that of theorem A.

We provide a proof here for completeness.

Since the tangent space of $G/B$ at the point $idB$ is 
$\mathfrak{g}/\mathfrak{b}$, the tangent bundle $\mathcal{T}_{G/B}$ 
is the homogeneous vector bundle $\mathcal{L}(\mathfrak{g}/\mathfrak{b})$ 
on $G/B$ associated to the $B$-module $\mathfrak{g}/\mathfrak{b}$.

Hence, it is sufficient to prove the following:

\begin{enumerate}
\item
$H^{i}(\tau, \mathfrak{g}/\mathfrak{b})=(0)$ for every $i\geq 1$. \\
\item
The adjoint representation $\mathfrak{g}$ of $G$ is a $B$-sub module of 
$H^{0}(\tau , \mathfrak{g}/\mathfrak{b})$ if and only if 
the set of semi stable points $X(\tau^{-1})_{T}^{ss}(\mathcal{L}_{\alpha_{0}})$ 
is non-empty. \\
\end{enumerate}

We prove this now.

Let $V_{1}:\mathfrak{g}$ and let $V_{2}:=\mathfrak{b}$.
The natural projection $\Pi:\mathfrak{g} \longrightarrow 
\mathfrak{g}/\mathfrak{b}$ of $B$-modules induces a 
homomorphism $\Pi_{\tau}:H^{0}(\tau, \mathfrak{g})
\longrightarrow H^{0}(\tau, \mathfrak{g}/\mathfrak{b})$ of $B$-modules.

{\it Proof of (1):} 

We have the short exact sequence    
 $(0)\longrightarrow \mathfrak{b} \longrightarrow \mathfrak{g} \longrightarrow 
\mathfrak{g}/\mathfrak{b} \longrightarrow (0)$ of $B$- modules.

Applying $H^{i}(\tau, -)$ to this short
exact sequence of $B$-modules,  we obtain 
the following long exact sequence of $B$-modules: 

{\it Observation 1}

$$\cdots H^{i}(\tau, \mathfrak{b}) \longrightarrow H^{i}(\tau, \mathfrak{g}) 
\longrightarrow H^{i}(\tau, \mathfrak{g}/\mathfrak{b})\longrightarrow 
H^{i+1}(\tau, \mathfrak{b}) \cdots$$

Onh the otherhand, by lemma(2.5(2)), we have 
$H^{i}(\tau, \mathfrak{g})=(0)$ for every $i\geq 1$. Further, by lemma(6.4), 
we have $H^{i+1}(\tau, \mathfrak{b})=(0)$ for every $i\geq 1$.
Applying this in the above long exact sequence of $B$-modules, we conclude 
that $H^{i}(\tau, \mathfrak{g}/\mathfrak{b})=(0)$ for every $i\geq 1$. 

This proves (1).

{\it Proof of 2:}

Since $H^{0}(\tau, \mathfrak{g})=\mathfrak{g}$, in order to prove (2),
it is sufficient to prove that the kerenel of the linear 
map $\Pi_{\tau}:H^{0}(\tau, \mathfrak{g}) \longrightarrow 
H^{0}(\tau, \mathfrak{g}/\mathfrak{b})$ is zero if and only if 
$X(\tau^{-1})_{T}^{ss}(\mathcal{L}_{\alpha_{0}})$ is non-empty.

We now show that the kerenel of the linear map 
$\Pi_{\tau}:H^{0}(\tau, \mathfrak{g}) \longrightarrow 
H^{0}(\tau, \mathfrak{g}/\mathfrak{b})$ is zero if and only 
if $X(\tau^{-1})_{T}^{ss}(\mathcal{L}_{\alpha_{0}})$ is non-empty.

By [10, lemma(2.1)], $X(\tau^{-1})_{T}^{ss}(\mathcal{L}_{\alpha_{0}})$ is 
non empty if and only if $\tau^{-1}(-\alpha_{0})\in R^{+}$.

Hence, we have 

{\it Observation 2}
$X(\tau^{-1})_{T}^{ss}(\mathcal{L}_{\alpha_{0}})$ is non empty
if and only if $-\alpha_{0}\in \tau(R^{+})$.

It is easy to see that  
$Ker(\Pi_{\tau})=H^{0}(\tau, \mathfrak{b})$. Hence, using lemma(6.4(1)), 
we see that $Ker(\Pi_{\tau})$ is a $B$-submodule of $\mathfrak{b}$. 

Since there is a unique $B$- stable line in $\mathfrak{b}$ and that is of 
weight $-\alpha_{0}$, we conclude that $Ker(P_{\tau})$ is a non-zero 
$B$-submodule of $\mathfrak{b}$ if and only if the $-\alpha_{0}$-weight 
space of  $H^{0}(\tau, \mathfrak{b})$ is non zero.

Hence, $Ker(P_{\tau})$ is non-zero if and only if 
$-\alpha_{0}\in \tau(R^{-})$.

Reformulating this statement, we have:

$Ker(P_{\tau})$ is zero if and only if 
$-\alpha_{0}\in \tau(R^{+})$.

Using {\it Observation 2}, we see that 
$Ker(P_{\tau})$ is zero if and only 
if the set of semi-stable points 
$X(\tau^{-1})_{T}^{ss}(\mathcal{L}_{\alpha_{0}})$ is non-empty.

This completes the proof of  (2).

\end{proof}

{\bf Remark:}
The second statement of theorem A does not hold for an arbitrary $\tau$ 
in case of $G$ is not simply laced. 

{\it Reason:}

For instance, let $G$ be of type $B_{2}$.
Let $\alpha_{1}$ and $\alpha_{2}$ be two simple roots such that  
$\langle \alpha_{1} , \alpha_{2} \rangle =-2$ and 
$\langle \alpha_{2} , \alpha_{1} \rangle =-1$.

We can take $\tau=s_{\alpha_{1}}s_{\alpha_{2}}s_{\alpha_{1}}$.
We see that $H^{1}(\tau, \mathfrak{b})$ is one dimensional representation 
$\mathbb{C}_{-(\alpha_{1}+\alpha_{2})}$ of $B$.

\end{document}